\documentclass{amsart}%
\usepackage{amsfonts, amsmath, amssymb}%
\usepackage{graphicx, epic, epsf, enumerate}
\usepackage{stmaryrd}

\input xy
\xyoption{all}

\newtheorem{theorem}{Theorem}
\theoremstyle{definition}

\newtheorem{corollary}{Corollary}

\newtheorem{example}{Example}

\newtheorem{proposition}{Proposition}
\newtheorem{remark}{Remark}

\numberwithin{equation}{section}
\newcommand{\brak}[1]{\langle #1\rangle}

\newcommand{\Z}{\mathbb{Z}}

\begin{document}

\title{Singular Links and Yang-Baxter State Models }

\author{Carmen Caprau}
\address{Department of Mathematics, California State University, Fresno, CA 93740, USA}
\email{ccaprau@csufresno.edu}
\urladdr{}
\author{Tsutomu Okano}
\address{Department of Mathematical Sciences, Carnegie Mellon University, Pittsburgh, PA 15213, USA}
\email{tsutomuo@andrew.cmu.edu}
\author{Danny Orton}
\address{Department of Mathematics, California State University, Fullerton, CA 92834, USA}
\email{ortondanny@gmail.com}

\date{}

\begin{abstract}
We employ a solution of the Yang-Baxter equation to construct invariants for knot-like objects. Specifically, we consider a Yang-Baxter state model for the $sl(n)$ polynomial of classical links and extend it to oriented singular links and balanced oriented 4-valent knotted graphs with rigid vertices. We also define a representation of the singular braid monoid into a matrix algebra, and seek conditions for extending further the invariant to contain topological knotted graphs. In addition, we show that the resulting Yang-Baxter-type invariant for singular links yields a version of the Murakami-Ohtsuki-Yamada state model for the $sl(n)$ polynomial for classical links.
\end{abstract}
\maketitle

\section{Introduction}\label{sec:intro}
The Yang-Baxter equation (YBE) was first introduced in the field of statistical mechanics. It takes its name from independent work of C.N. Yang in 1968 and R.J. Baxter in 1971. It depends on the idea that in some scattering situations, particles may preserve their momentum in price of changing their quantum internal states. One form of the YBE states that a matrix  $R$, acting on two of the three objects, satisfies
\[ (R \otimes I) (I \otimes R) (R \otimes I) = (I \otimes R) (R \otimes I) (I \otimes R), \] 
in which case $R$ is called a solution of the YBE. This equation shows up when working with braid groups (in which case $R$ corresponds to swapping two braid strands) and when discussing invariants for knots and links. A relationship between the YBE and polynomial invariants of links was implicitly revealed by V. Jones in his seminal paper~\cite{J}, introducing a one-variable polynomial of links via a study of finite dimensional von Neumann algebras. The Jones polynomial was almost immediately generalized by J. Hoste, A. Ocneanu, K. Millett, W.B.R. Lickorish, P. Freyd, D. Yetter, J. Przytycki and P. Traczyk to a two-variable polynomial for oriented links (see~\cite{HOMFLY, PT}), the so-called HOMFLY-PT polynomial, which can be defined via a Conway-type skein relation. Using an analogous geometric procedure, L. Kauffman introduced a two-variable polynomial invariant of regular isotopy for unoriented knots and links (see~\cite{K1, K2}). 

Jones showed that the HOMFLY-PT polynomial can be constructed using explicit matrix representations of Hecke algebras, introduced in works on quantum scattering method and related to the YBE. Using Yang-Baxter operators and so-called EYB-operators (that is, enhanced Yang-Baxter operators), V. Turaev~\cite{Tu} associated with each EYB-operator an isotopy invariant of links, and showed that for some special EYB-operators, the corresponding invariants are equivalent to the HOMFLY-PT polynomial and the two-variable Kauffman polynomial.

In his excellent book~\cite{Ka}, L. Kauffman provides Yang-Baxter state models for certain polynomial invariants for links. These state models make use of solutions of the YBE.

In the recent years, there has been a great interest in the study of knot-like objects, including singular links, knotted graphs, and virtual knots. A \textit{knotted graph} is an embedding of a graph in three-dimensional space, and a \textit{singular link} is an immersion of a disjoint union of circles into three-dimensional space, which admits only finitely many singularities that are all transverse double points. The goal of this paper is to extend Kauffman's Yang-Baxter state model for the $sl(n)$ polynomial (which is a one-variable specialization of the HOMFLY-PT polynomial) to oriented singular links and 4-valent knotted graphs. Along the way, we define a representation of the singular braid monoid. Moreover, we arrive at certain skein relations for planar 4-valent graphs, relations which remind us of the Murakami-Ohtsuki-Yamada (MOY)~\cite{MOY} state model for the $sl(n)$-link invariant. These relations assign well-defined polynomials to planar 4-valent graphs by recursive formulas defined entirely in the category of planar graphs.

We remark that there is an EYB-operator (as in~\cite{Tu}) associated with the regular isotopy polynomial invariant for singular links constructed here. However, we focus in this paper on Kauffman's combinatorial approach to Yang-Baxter state models.

\textit{Organization of the paper.} In Section~\ref{sec:YB-sl(n)poly} we recall the Yang-Baxter state model for the regular isotopy version of the $sl(n)$ polynomial and introduce some notation. In Section~\ref{sec:inv-singlinks} we extend this state model to a regular isotopy invariant for singular links (based on a solution of the YBE) and discuss some of its properties. Then we use the resulting state model to construct in Section~\ref{sec:repres} a representation of the singular braid monoid into a matrix algebra over the ring $\Z[q, q^{-1}]$. Section~\ref{sec:MOYrelations} is devoted to showing that our polynomial invariant for singular links yields a version of the MOY state model for the $sl(n)$ polynomial. Finally, in Section~\ref{sec:knotted graphs} we extend further our polynomial invariant so that it contains balanced oriented 4-valent knotted graphs with rigid vertices. We also find a numerical invariant of 4-valent topological knotted graphs.

\section{A Yang-Baxter model for the $sl(n)$ polynomial}\label{sec:YB-sl(n)poly}

In this section we briefly review the Yang-Baxter state model for the $sl(n)$ polynomial introduced by Kauffman~\cite{Ka}. 
Given a link diagram $D$, label its edges with \textit{spins} from the equally spaced index set $I_n=\{ 1-n,3-n, \dots,n-3,n-1\}$, for $n \in \mathbb{Z}, n \geq 2$, as follows: replace each crossing in $D$ by either a \textit{decorated splice}
\[\hspace{0.2cm} \raisebox{-17pt}{  \includegraphics[height=.5in]{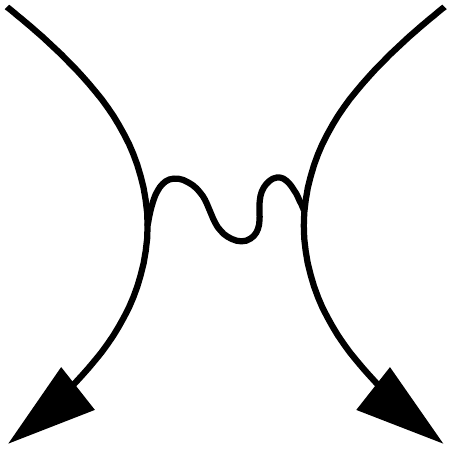}} \hspace{1cm} \raisebox{-17pt}{  \includegraphics[height=.5in]{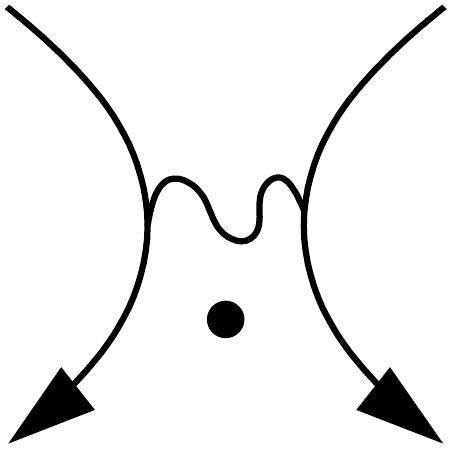}}
 \hspace{1cm} \raisebox{-17pt}{  \includegraphics[height=.5in]{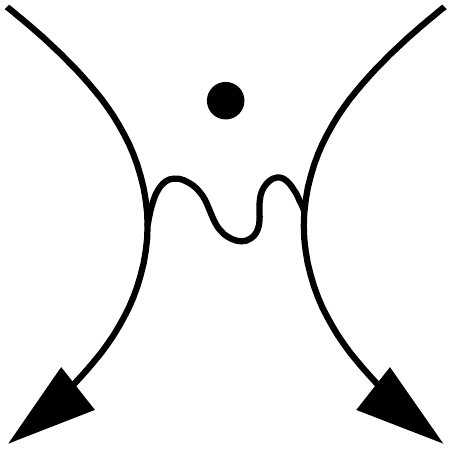}} \]
or by a \textit{flat crossing}
\[\raisebox{-17pt}{  \includegraphics[height=.5in]{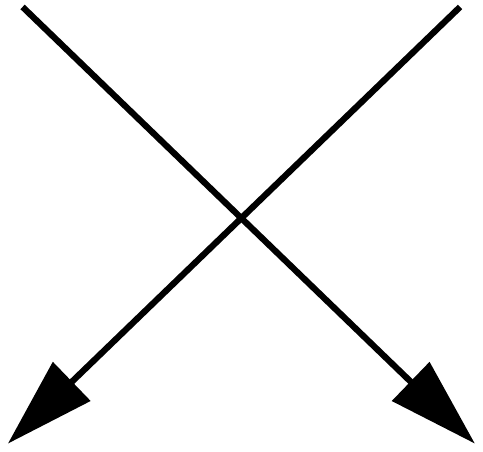} }\]
and label the resulting diagram $\sigma$ with spins from the set $I_n$, so that each loop in  $\sigma$ has constant spin, and so that the spins satisfy the following rules:
\vspace{-0.5cm}
\[
\begin{picture}(50,50)
    \raisebox{-17pt}{  \includegraphics[height=.5in]{eqsplit}}
      \put(-40, 20){\fontsize{10}{11}$ a$}
  \put(0, 20){\fontsize{10}{11}$ b$}
     \,\,\, $\Longrightarrow \,\,\, a = b$
   \end{picture} 
      \hspace{3cm}
\begin{picture}(50,50)
      \raisebox{-17pt}{  \includegraphics[height=.5in]{ltsplit}}
       \put(-40, 20){\fontsize{10}{11}$ a$}
  \put(0, 20){\fontsize{10}{11}$ b$}
       \,\,\,  $\Longrightarrow \,\,\, a < b$
  \end{picture} 
  \]
   \vspace{-0.8cm}
\[
\begin{picture}(50,50)
      \raisebox{-17pt}{  \includegraphics[height=.5in]{gtsplit}}
      \put(-40, 20){\fontsize{10}{11}$ a$}
  \put(0, 20){\fontsize{10}{11}$ b$}
      \,\,\, $\Longrightarrow \,\,\, a > b$
      \end{picture} 
      \hspace{3cm}
\begin{picture}(50,50)
      \raisebox{-17pt}{  \includegraphics[height=.5in]{flat}}
       \put(-40, 20){\fontsize{10}{11}$ a$}
  \put(0, 20){\fontsize{10}{11}$ b$}
      \,\,\,  $\Longrightarrow \,\,\, a \neq b$
      \end{picture} 
      \]
\vspace{0.2cm}

The result is a \textit{state} of $D$. Notice that some of the states will have incompatible labels (spins) and thus are discarded.

Associate to each state $\sigma$ a polynomial $\brak{ \sigma} \,\in \mathbb{Z}[q,q^{-1}]$ given by:

\begin{eqnarray}\label{eq:state-eval}
\brak{\sigma}\, = q^{||\sigma ||}, \,\,\,\,  ||\sigma||=\sum_{l} \text{rot}(l)\cdot \text{label}(l),
\end{eqnarray}
where the sum is taken over all components $l$ in $\sigma$, $\text{label}(l)$ is the spin assigned to the loop $l$, and where $\text{rot}(l)$ is the \textit{rotation number} of $l$ given by:
\[\text{rot}\left( \raisebox{-7pt}{ \includegraphics[height=.3in]{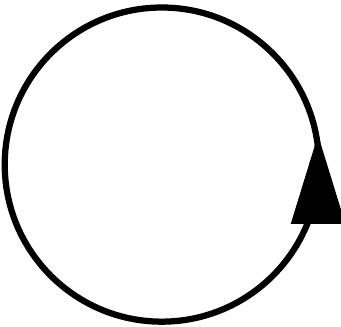}}  \right) =1, \,\,\,
\text{rot}\left(\raisebox{-7pt}{ \includegraphics[height=.3in]{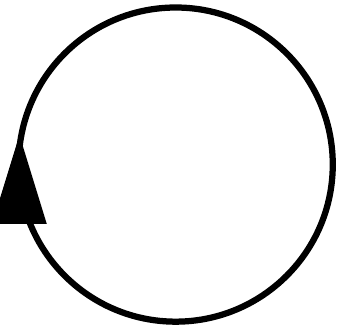}}  \right) =-1.\]

\begin{example} For the state $\sigma_1$ below, $\brak{\sigma_1} = q^{2a-b}$. On the other hand, the state $\sigma_2$ will have incompatible spins for any choice of labels, and thus it is discarded. Equivalently,  we set $\brak{\sigma_2} = 0$.
   \[ \sigma_1=  \raisebox{-33pt}{ \includegraphics[height=1in]{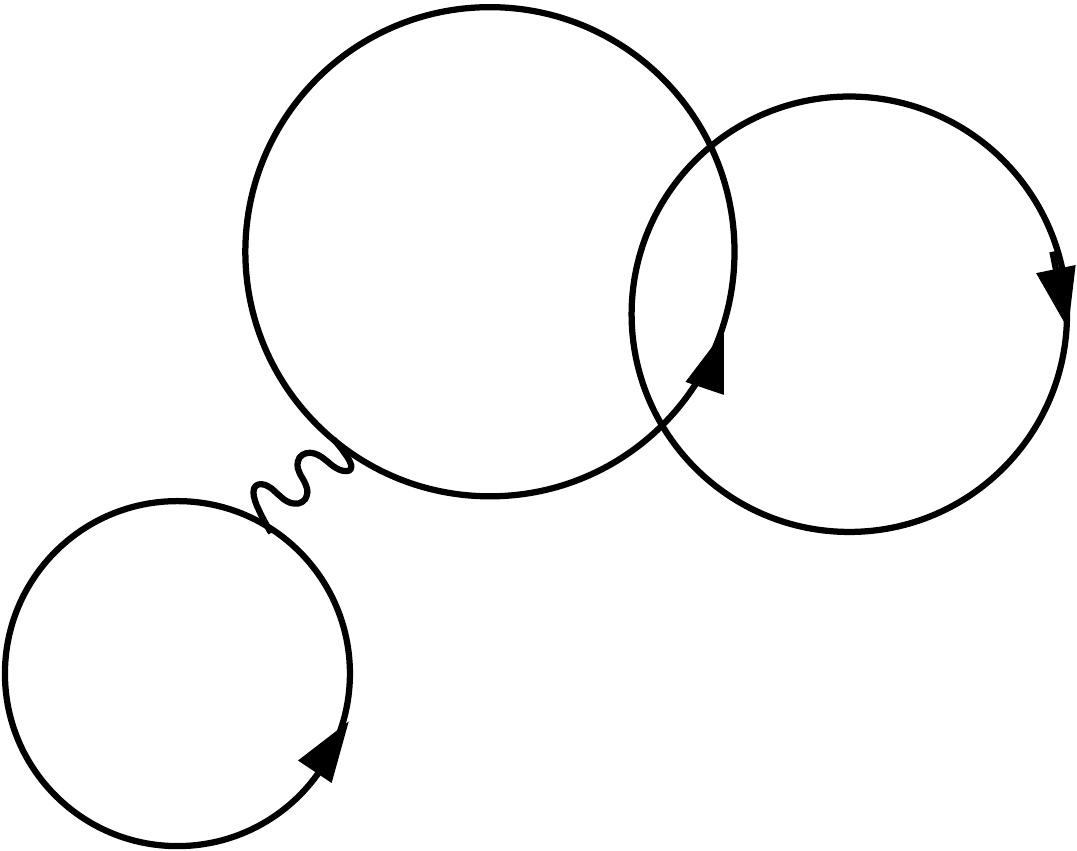}}
    \put(3, 10){\fontsize{10}{11}$ b$}
     \put(-33, 32){\fontsize{10}{11}$ \neq$}
       \put(-40, -7){\fontsize{10}{11}$ \neq$}
     \put(-67, 20){\fontsize{10}{11}$ a$}
       \put(-60, -25){\fontsize{10}{11}$ a$}\hspace{1.5cm}\raisebox{-17pt}{ \includegraphics[height=0.6in]{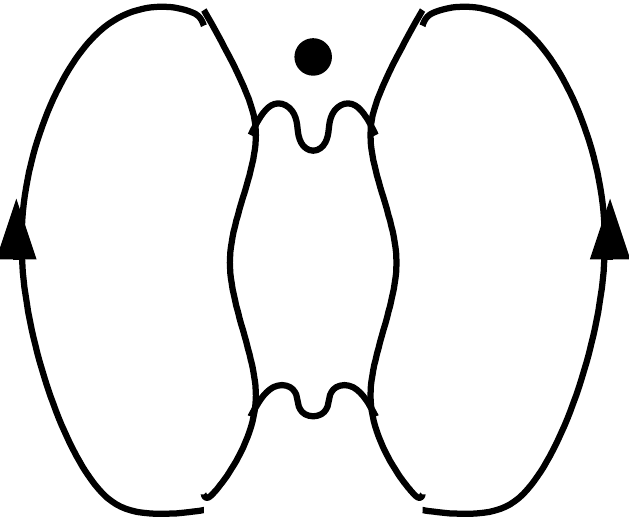}} = \sigma_2
 \]
\end{example}

The \textit{$sl(n)$ polynomial} of the link diagram $D$ is given by:

\begin{eqnarray}\label{eq:polyn-eval}
\left< D \right>= \sum_{\sigma} a_{\sigma}\, \brak{\sigma} = \sum_{\sigma} a_{\sigma}\, q^{||\sigma||}, 
\end{eqnarray}
where the sum is taken over all states $\sigma$ of $D$ and where $a_{\sigma} $ is the product of the weights associated with a state $\sigma$ according to the skein relations given in Figure~\ref{fig:crossings}.
\begin{figure}[ht]
\[
\left<\,\,
\begin{picture}(40,20)
      \raisebox{-13pt}{ \includegraphics[height=.4in]{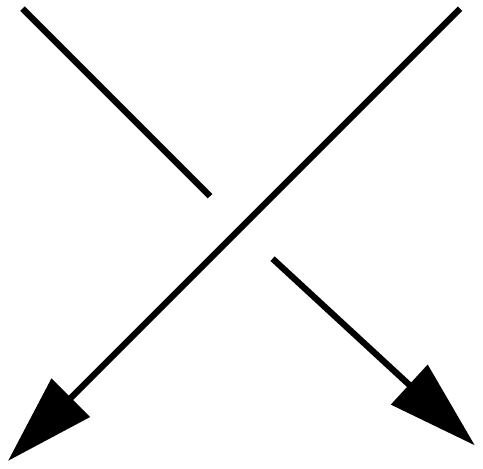}}
                 \put(-35, 17){\fontsize{10}{11}$ a$}
  \put(0, 17){\fontsize{10}{11}$ b$}
   \put(-35, -20){\fontsize{10}{11}$ c$}
     \put(0, -20){\fontsize{11}{11}$ d$}
                     \end{picture} 
             \right>
               = (q-q^{-1})\left<\,\,
\begin{picture}(40,20)
      \raisebox{-13pt}{ \includegraphics[height=.4in]{ltsplit}}
       \put(-35, 17){\fontsize{10}{11}$ a$}
  \put(0, 17){\fontsize{10}{11}$ b$}
    \put(-35, -20){\fontsize{10}{11}$ c$}
     \put(0, -20){\fontsize{10}{11}$ d$}
      \end{picture}
 \right>
      +q \left<\,\,
\begin{picture}(40,20)
      \raisebox{-13pt}{ \includegraphics[height=.4in]{eqsplit}}
     \put(-35, 17){\fontsize{10}{11}$ a$}
  \put(0, 17){\fontsize{10}{11}$ b$}
    \put(-35, -20){\fontsize{10}{11}$ c$}
     \put(0, -20){\fontsize{10}{11}$ d$}
      \end{picture}  
 \right>
      + \left<\,\,
\begin{picture}(40,20)
     \raisebox{-13pt}{  \includegraphics[height=.4in]{flat}}
 \put(-35, 17){\fontsize{10}{11}$ a$}
  \put(0, 17){\fontsize{10}{11}$ b$}
    \put(-35, -20){\fontsize{10}{11}$ c$}
     \put(0, -20){\fontsize{10}{11}$ d$}
              \put(-18, 7){\fontsize{10}{11}$ \neq$}
                    \end{picture} 
 \right> 
       \]
 \vspace{0.5cm}
 \[
\left<\,\,
\begin{picture}(40,20)
      \raisebox{-13pt}{ \includegraphics[height=.4in]{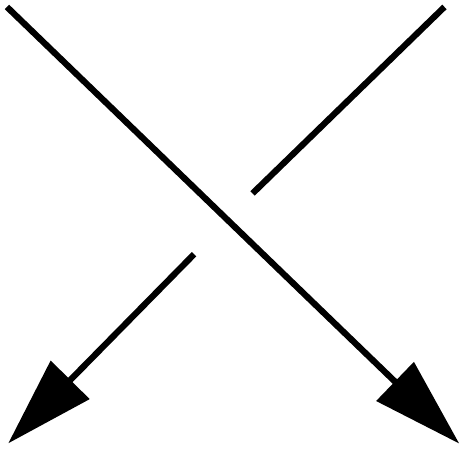}}
                 \put(-35, 17){\fontsize{10}{11}$ a$}
  \put(0, 17){\fontsize{10}{11}$ b$}
   \put(-35, -20){\fontsize{10}{11}$ c$}
     \put(0, -20){\fontsize{10}{11}$ d$}
                     \end{picture} 
         \right>     
                = (q^{-1}-q)\left<\,\,
\begin{picture}(40,20)
      \raisebox{-13pt}{ \includegraphics[height=.4in]{gtsplit}}
       \put(-35, 17){\fontsize{10}{11}$ a$}
  \put(0, 17){\fontsize{10}{11}$ b$}
    \put(-35, -20){\fontsize{10}{11}$ c$}
     \put(0, -20){\fontsize{10}{11}$ d$}
      \end{picture}
 \right>
      +q^{-1} \left<\,\,
\begin{picture}(40,20)
      \raisebox{-13pt}{ \includegraphics[height=.4in]{eqsplit}}
     \put(-35, 17){\fontsize{10}{11}$ a$}
  \put(0, 17){\fontsize{10}{11}$ b$}
    \put(-35, -20){\fontsize{10}{11}$ c$}
     \put(0, -20){\fontsize{10}{11}$ d$}
      \end{picture}  
 \right>
      + \left<\,\,
\begin{picture}(40,20)
     \raisebox{-13pt}{  \includegraphics[height=.4in]{flat}}
 \put(-35, 17){\fontsize{10}{11}$ a$}
  \put(0, 17){\fontsize{10}{11}$ b$}
    \put(-35, -20){\fontsize{10}{11}$ c$}
     \put(0, -20){\fontsize{10}{11}$ d$}
              \put(-18, 7){\fontsize{10}{11}$ \neq$}
                    \end{picture} 
     \right>.
\]
\caption{ Crossings decomposition}\label{fig:crossings}
\end{figure}

The diagrams in the two sides of the skein relations in Figure~\ref{fig:crossings} represent parts of bigger link diagrams that are the same, except in a small neighborhood where they differ as shown in the given relation. 

According to the rules in Figure~\ref{fig:crossings} and due to the requirement that each loop in a state $\sigma$ with $\brak{\sigma} \neq 0$ has constant spin, it follows that the evaluation of a crossing is non-zero only when the spins $a, b, c$ and $d$ associated with the four endpoints of the crossing satisfy the \textit{conservation law} $a+b =c+d$. In particular, the evaluation of a crossing is non-zero if and only if $a=c$ and $b = d$ or $a = d$ and $b = c$.

We can arrive at the $sl(n)$ polynomial $\brak{D}$ by interpreting link diagrams as \textit{abstract tensor diagrams}. An oriented link diagram $D$ can be decomposed with respect to a height function into minima (creations), maxima (annihilations) and crossings (interactions), as illustrated in Figure~\ref{fig:tensor}. That is, the diagram $D$ is constructed from interconnected maxima, minima and crossings, and we want to associate to them square matrices with entries in $\mathbb{Z}[q, q^{-1}]$. 

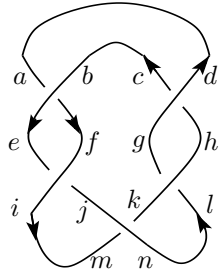
\begin{figure}[ht]
\includegraphics[height =1.4in, width = 1in]{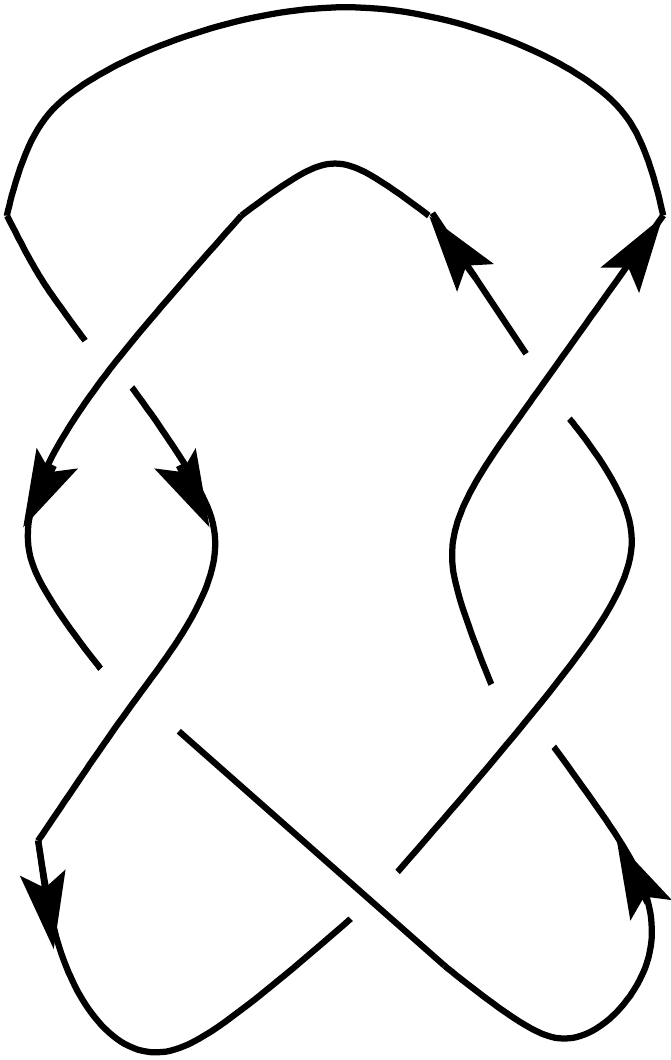}
  \put(-75,70){\fontsize{10}{11}$ a$}
  \put(-49, 70){\fontsize{10}{11}$ b$}
  \put(-30, 70){\fontsize{10}{11}$ c$}
  \put(-3, 70){\fontsize{10}{11}$ d$}
   \put(-77, 45){\fontsize{10}{11}$ e$}
  \put(-48, 45){\fontsize{10}{11}$ f$}
  \put(-30, 46){\fontsize{10}{11}$ g$}
  \put(-3, 45){\fontsize{10}{11}$ h$}
   \put(-76, 20){\fontsize{10}{11}$ i$}
  \put(-51, 19){\fontsize{10}{11}$ j$}
  \put(-32, 23){\fontsize{10}{11}$ k$}
  \put(-2, 21){\fontsize{10}{11}$ l$}
   \put(-46, 0){\fontsize{10}{11}$ m$}
  \put(-28, 0){\fontsize{10}{11}$ n$}
\caption{A diagram as an abstract tensor diagram}\label{fig:tensor}
\end{figure}

We associate the symbols $R^{ab}_{cd}$ and $\overline{R}^{ab}_{cd}$ to the positive and negative crossings, respectively:
\[
R^{ab}_{cd}=
\begin{picture}(50,30)
      \raisebox{-13pt}{ \includegraphics[height=.4in]{pcross}}
                \put(-37, 17){\fontsize{10}{11}$ a$}
  \put(0, 17){\fontsize{10}{11}$ b$}
    \put(-35, -20){\fontsize{10}{11}$ c$}
     \put(0, -20){\fontsize{10}{11}$ d$}
                \end{picture} 
                 \hspace{2cm}
\overline{R}^{ab}_{cd}=
\begin{picture}(50,30)
      \raisebox{-13pt}{ \includegraphics[height=.4in]{ncross}}
                   \put(-37, 17){\fontsize{10}{11}$ a$}
  \put(0, 17){\fontsize{10}{11}$ b$}
    \put(-35, -20){\fontsize{10}{11}$ c$}
     \put(0, -20){\fontsize{10}{11}$ d$}
                     \end{picture} \vspace{0.7cm}
\]
where $a, b, c, d \in I_n$.
With these conventions, the skein relations in Figure~\ref{fig:crossings} can be rewritten as follows:
\begin{eqnarray*}
R^{ab}_{cd}&=&(q-q^{-1}) [a<b]  \text{ } \delta^a_c \text{ }\delta^b_d  + q [a=b] \text{ }  \delta^a_c  \text{ } \delta^b_d  +[a\neq b]  \text{ }  \delta^a_d  \text{ } \delta^b_c\\
\overline{R}^{ab}_{cd}&=&(q^{-1}-q) [a>b]  \text{ } \delta^a_c  \text{ } \delta^b_d  + q^{-1} [a=b]  \text{ } \delta^a_c  \text{ } \delta^b_d  +[a\neq b]   \text{ } \delta^a_d  \text{ } \delta^b_c 
\end{eqnarray*}

where
\[[a=b] \delta^a_c \delta^b_d= \begin{picture}(60,30)
      \raisebox{-13pt}{  \includegraphics[height=.4in]{eqsplit}}
      \put(-37, 17){\fontsize{10}{11}$ a$}
  \put(0, 17){\fontsize{10}{11}$ b$}
    \put(-35, -20){\fontsize{10}{11}$ c$}
     \put(0, -20){\fontsize{10}{11}$ d$}
           \end{picture} 
      [a<b] \delta^a_c \delta^b_d= \begin{picture}(60,30)
      \raisebox{-13pt}{  \includegraphics[height=.4in]{ltsplit}}
       \put(-37, 17){\fontsize{10}{11}$ a$}
  \put(0, 17){\fontsize{10}{11}$ b$}
    \put(-37, -20){\fontsize{10}{11}$ c$}
     \put(0, -20){\fontsize{10}{11}$ d$}
      \end{picture}  \]
    
      \[[a>b] \delta^a_c \delta^b_d= \begin{picture}(60,30)
      \raisebox{-13pt}{  \includegraphics[height=.4in]{gtsplit}}
     \put(-37, 17){\fontsize{10}{11}$ a$}
  \put(0, 17){\fontsize{10}{11}$ b$}
    \put(-35, -20){\fontsize{10}{11}$ c$}
     \put(0, -20){\fontsize{10}{11}$ d$}
           \end{picture}
       [a\neq b]   \delta^a_d \delta^b_c= \begin{picture}(60,30)
      \raisebox{-13pt}{  \includegraphics[height=.4in]{flat}}
       \put(-37, 17){\fontsize{10}{11}$ a$}
  \put(0, 17){\fontsize{10}{11}$ b$}
    \put(-37, -17){\fontsize{10}{11}$ c$}
     \put(0, -20){\fontsize{10}{11}$ d$}
       \put(-18, 7){\fontsize{10}{11}$ \neq$}
      \end{picture} 
\]
and where $[P] = \begin{cases} 1\,\,\, \text{ if P is true}\\ 0 \,\,\,\,  \text{if P is false}
\end{cases}
$ and \, $  \hspace{0.1cm}
             \begin{picture}(15,40) {\raisebox{-13pt}{ \includegraphics[height=.4in]{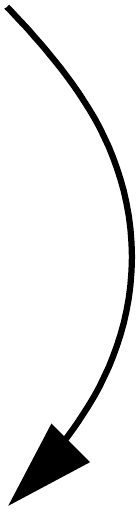}} } 
  \put(-15, 17){\fontsize{10}{11}$ a$}
     \put(-15, -17){\fontsize{10}{11}$ c$}
      \end{picture}
= \delta^a_c = \begin{cases} 1 \,\,\, \text{if} \,\,\, a=c \\0 \,\,\, \text{if} \,\,\, a \neq c. \end{cases}$
\vspace{0.4cm}

We associate the symbols $\overrightarrow{M}^{ab}, \overleftarrow{M}^{ab}$ and $\overrightarrow{M}_{ab}, \overleftarrow{M}_{ab}$ to oriented minima and maxima, respectively, and we put
\[\overrightarrow{M}^{ab} = \begin{picture}(40,15) \reflectbox{\raisebox{7pt}{\includegraphics[width=.15in, angle=270]{arc}}}
\put(-40, 10){\fontsize{10}{11}$ a$}
  \put(0, 10){\fontsize{10}{11}$ b$}
      \end{picture} \,\,= q^{a/2} \delta^{a,b} \hspace{1cm}
      \overleftarrow{M}^{ab} = \begin{picture}(40,15)\raisebox{7pt}{\includegraphics[width=.15in, angle=270]{arc}}
\put(-40, 10){\fontsize{10}{11}$ a$}
  \put(0, 10){\fontsize{10}{11}$ b$} 
      \end{picture}\,\, =  q^{-a/2} \delta^{a,b}\]  
          
 \[\overleftarrow{M}_{ab} = \begin{picture}(40,15) \reflectbox{\raisebox{1pt}{\includegraphics[width=.15in, angle=90]{arc}}}
\put(-40, -7){\fontsize{10}{11}$ a$}
  \put(0, -7){\fontsize{10}{11}$ b$}
      \end{picture}\,\, = q^{a/2} \delta^{a,b} \hspace{1cm}
      \overrightarrow{M}_{ab} = \begin{picture}(40,15) \raisebox{1pt}{\includegraphics[width=.15in, angle=90]{arc}}
\put(-40, -7){\fontsize{10}{11}$ a$}
  \put(0, -7){\fontsize{10}{11}$ b$}
      \end{picture} \,\,=  q^{-a/2} \delta^{a,b}
        \vspace{0.5cm}\]
where $ \delta^{a,b} =  \begin{cases} 1, \,\,\, a=b \\0, \,\,\, a \neq b. \end{cases}$

Therefore, for the diagram $D$ in Figure~\ref{fig:tensor}, the evaluation $\brak{D}$ is given by the following sum of products of matrix entries:
\[ \brak{D} = \sum_{a, b, \dots, n \in  I_n} \overleftarrow{M}_{ad} \overleftarrow{M}_{bc}R^{ab}_{ef} R^{ef}_{ij} \overrightarrow{M}^{im} R^{mj}_{nk} \overrightarrow{M}^{nl} R^{lk}_{hg} R^{hg}_{dc} \]
where the sum is over all possible choices of indices (spins from $I_n$) in the expression.

It is important to note that the above conventions yield the necessary loop value, namely $[n] = \frac{q^n -q^{-n}}{q-q^{-1}}$, where $[n]$ is the \textit{quantum integer} $n$. On one hand,
\[ \left< \raisebox{-9pt}{  \includegraphics[height=.3in]{pcirc}} \,\right> = \sum_{a \in I_n} \left< \raisebox{-9pt}{  \includegraphics[height=.3in]{pcirc}} \,\right>  \put(-39, 1){\fontsize{10}{11}$a$}  = \sum_{a\in I_n} q^a = [n],  \]
and, on the other hand,
\[\sum_{a \in I_n} \left< \raisebox{-9pt}{  \includegraphics[height=.3in]{pcirc}} \,\right>  \put(-39, 1){\fontsize{10}{11}$a$} = \sum_{a\in I_n} \left( \sum_{b \in I_n} \overleftarrow{M}_{ab}\overrightarrow{M}^{ab}\right) = \sum_{a\in I_n} q^a.    \]
Moreover, the loop value stays the same if the circle is clockwise oriented, since
 \[\sum_{a\in I_n} q^a = q^{1-n} + q^{3-n} + \dots + q^{n-3} + q^{n-1} = \sum_{a\in I_n} q^{-a}.\]
Observe that the creation and annihilation matrices satisfy: 
\[  \sum_{i \in  I_n} \overrightarrow{M}^{ai}\overrightarrow{M}_{ib} = \delta^a_b = \sum_{i \in  I_n} \overleftarrow{M}_{bi}\overleftarrow{M}^{ia} \]
\[ \sum_{i \in  I_n} \overleftarrow{M}^{ai}\overleftarrow{M}_{ib} = \delta^b_a = \sum_{i \in  I_n} \overrightarrow{M}_{bi}\overrightarrow{M}^{ia}\]
which correspond, respectively, to the following planar isotopies (canceling pairs of maxima and minima):
\[
\raisebox{-24pt}{\includegraphics[height=0.6in]{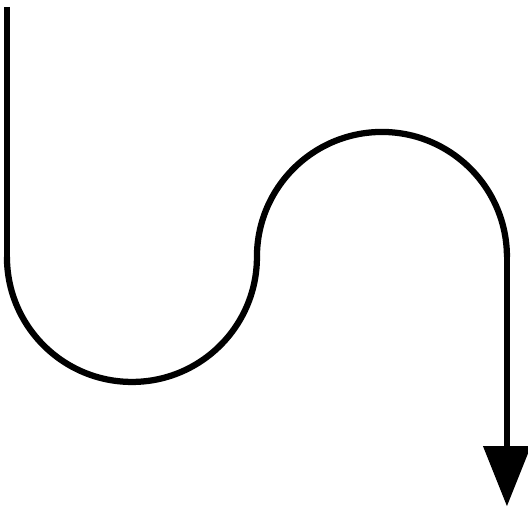}}\,\,\, = \,\,\, \raisebox{-24pt}{\includegraphics[height=0.6in]{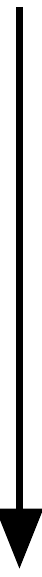}} \,\,\, =\,\,\,  \reflectbox{\raisebox{-24pt}{\includegraphics[height=0.6in]{s-move}}}
\put(-65,-23){\fontsize{10}{11}$b$}
\put(-93,-23){\fontsize{10}{11}$b$}
\put(-39,-23){\fontsize{10}{11}$b$}
\put(-67, 16){\fontsize{10}{11}$a$}
\put(-149, 16){\fontsize{10}{11}$a$}
\put(3, 16){\fontsize{10}{11}$a$}
\put(-125, -2){\fontsize{10}{11}$i$}
\put(-20, -2){\fontsize{10}{11}$i$} 
\hspace{1.7cm}
\raisebox{25pt}{\includegraphics[height=0.6in, angle = 180]{s-move}}\,\,\, = \,\,\, \raisebox{25pt}{\includegraphics[height=0.6in, angle=180]{arc-2}} \,\,\, =\,\,\,  \reflectbox{\raisebox{25pt}{\includegraphics[height=0.6in, angle=180]{s-move}}}
\put(-65,-19){\fontsize{10}{11}$b$}
\put(-93,-19){\fontsize{10}{11}$b$}
\put(-39,-19){\fontsize{10}{11}$b$}
\put(-66, 18){\fontsize{10}{11}$a$}
\put(-149, 18){\fontsize{10}{11}$a$}
\put(3, 19){\fontsize{10}{11}$a$}
\put(-125, 0){\fontsize{10}{11}$i$}
\put(-20, 0){\fontsize{10}{11}$i$}
\]
The matrices $R$ and $\overline{R}$ satisfy the \textit{channel unitarity}
\[
\sum_{i,i \in I_n}R^{ab}_{ij} \overline{R}^{ij}_{cd} =
 \begin{picture}(50,30) \raisebox{-13pt}{  \includegraphics[height=.4in]{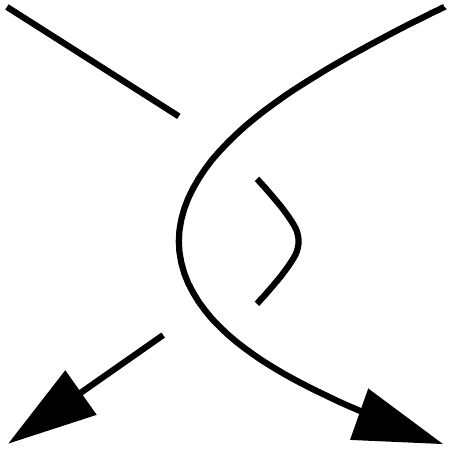} }
  \put(-37, 17){\fontsize{10}{11}$ a$}
  \put(0, 17){\fontsize{10}{11}$ b$}
    \put(-37, -20){\fontsize{10}{11}$ c$}
     \put(0, -20){\fontsize{10}{11}$ d$}
     \put(-30, 0){\fontsize{10}{11}$ i$}
  \put(-10, 0){\fontsize{10}{11}$ j$}
      \end{picture}
       \sim \hspace{0.3cm}
      \begin{picture}(60,30)
      \raisebox{-13pt}{  \includegraphics[height=.4in]{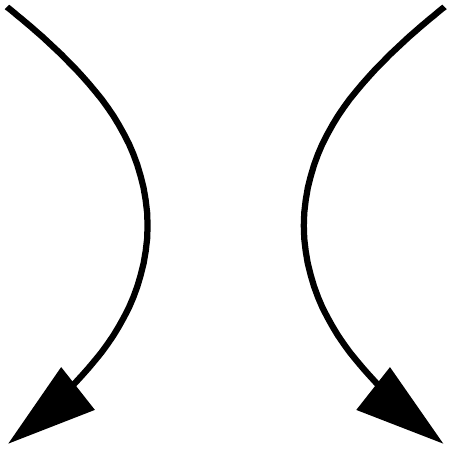}}
         \put(-35, 17){\fontsize{10}{11}$ a$}
  \put(0, 17){\fontsize{10}{11}$ b$}
    \put(-35, -20){\fontsize{10}{11}$ c$}
     \put(0, -20){\fontsize{10}{11}$ d$}
      \end{picture}
       =\delta^a_c \delta^b_d 
          \vspace{0.7cm}
\]
and \textit{cross-channel unitarity}
\[
\sum_{i, j \in I_n}\overline{R}^{ia}_{jb} R^{jd}_{ic} = 
\begin{picture}(50,30) \raisebox{-13pt}{ \includegraphics[height=.4in]{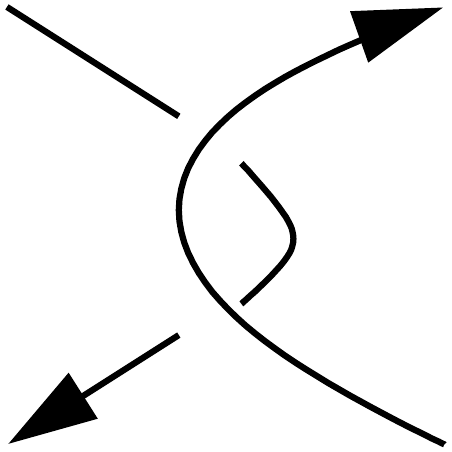} }
   \put(-37, 17){\fontsize{10}{11}$ a$}
  \put(0, 17){\fontsize{10}{10}$ b$}
    \put(-37, -20){\fontsize{10}{11}$ c$}
     \put(0, -20){\fontsize{10}{11}$ d$}
     \put(-30, 0){\fontsize{10}{11}$ i$}
  \put(-10, 0){\fontsize{10}{11}$ j$}
           \end{picture}
       \sim \hspace{0.3cm}
    \begin{picture}(60,30)  \raisebox{-13pt}{ \includegraphics[height=.4in]{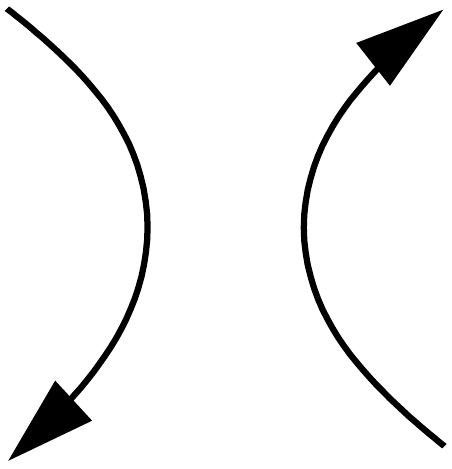}}  
     \put(-35, 17){\fontsize{12}{11}$ a$}
  \put(0, 17){\fontsize{10}{11}$ b$}
    \put(-35, -20){\fontsize{10}{11}$ c$}
     \put(0, -20){\fontsize{10}{11}$ d$}
           \end{picture}
       =\delta^a_c \delta^d_b. 
       \vspace{0.7cm}
\]
Moreover, we have that 
\[
\sum_{i,j,k\in I}R^{ab}_{ij} R^{jc}_{kf} R^{ik}_{de} =
\begin{picture}(60,40) \raisebox{-35pt}{ \includegraphics[height=1in]{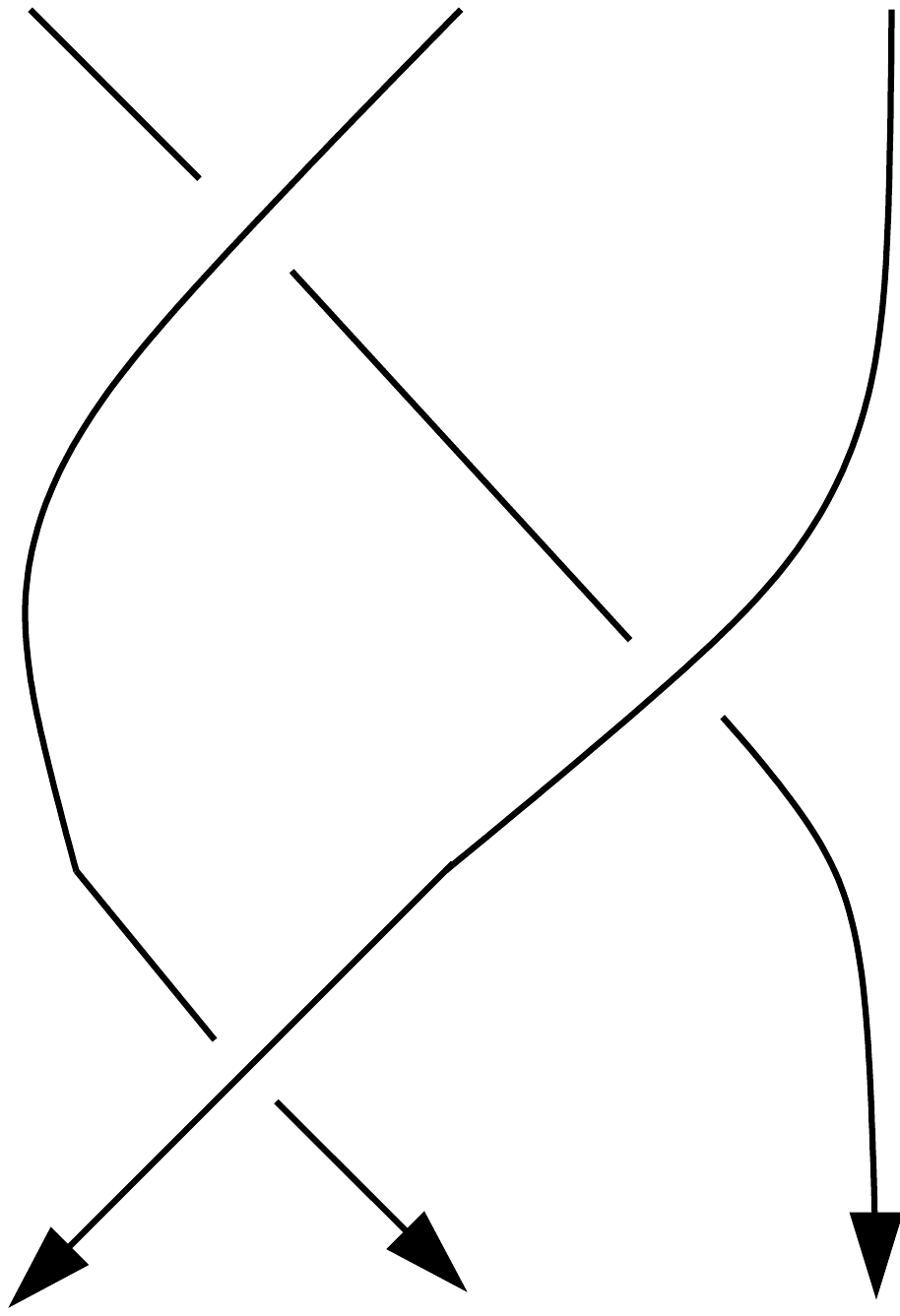} }
\put(-55, 38){\fontsize{10}{11}$ a$}
  \put(-32, 38){\fontsize{10}{11}$ b$}
  \put(-8, 38){\fontsize{10}{11}$ c$}
  \put(-50, 0){\fontsize{10}{11}$ i$}
   \put(-27, -18){\fontsize{10}{11}$ k$}
   \put(-28, 15){\fontsize{10}{11}$ j$}
    \put(-55, -44){\fontsize{10}{11}$ d$}
     \put(-8, -44){\fontsize{10}{11}$ f$}
       \put(-32, -44){\fontsize{10}{11}$ e$}
      \end{picture}
       \sim 
\begin{picture}(60,40)  \raisebox{-35pt}{ \includegraphics[height=1in]{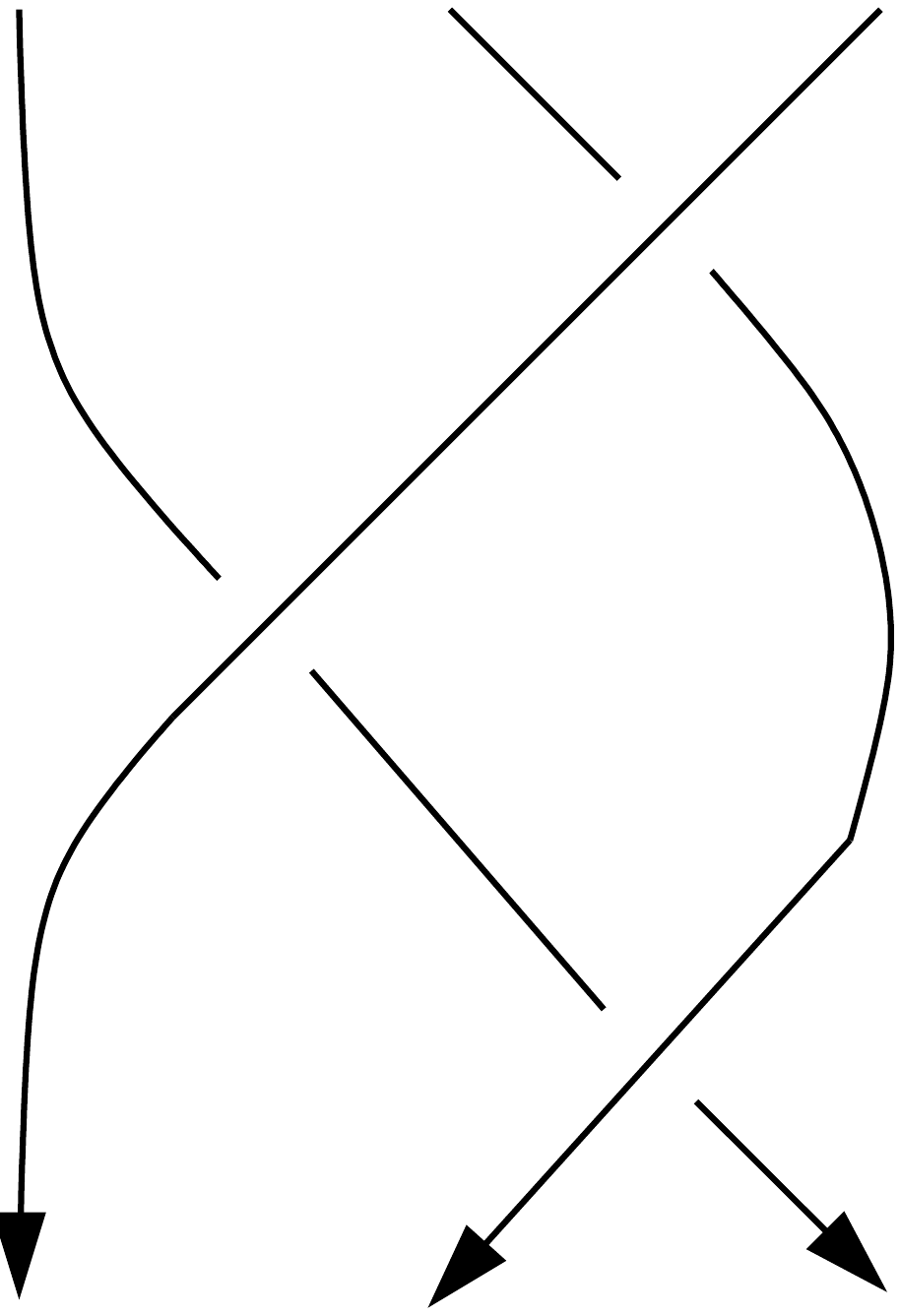} }
\put(-55, 38){\fontsize{10}{11}$ a$}
  \put(-32, 38){\fontsize{10}{11}$ b$}
  \put(-8, 38){\fontsize{10}{11}$ c$}
  \put(-12, 0){\fontsize{10}{11}$ j$}
   \put(-34, -18){\fontsize{10}{11}$ k$}
   \put(-31, 15){\fontsize{10}{11}$ i$}
    \put(-55, -44){\fontsize{10}{11}$ d$}
     \put(-8, -44){\fontsize{10}{11}$ f$}
       \put(-32, -44){\fontsize{10}{11}$ e$}
             \end{picture}
      =\sum_{i,j,k\in I}R^{bc}_{ij} R^{ai}_{dk} R^{kj}_{ef}.
      \vspace{1.4cm}
    \]
The latter relation is the \textit{Yang-Baxter equation} (YBE):
$$\sum_{i,j,k\in I}R^{ab}_{ij} R^{jc}_{kf} R^{ik}_{de} = \sum_{i,j,k \in I}R^{bc}_{ij} R^{ai}_{dk} R^{kj}_{ef}. $$
That is, the $R$ matrix as defined above is a solution of the YBE. Similarly, the matrix $\overline{R}$ is a solution of the YBE. Moreover, 

\begin{eqnarray*}
\left<\,\,
\begin{picture}(40,20)
      \raisebox{-13pt}{ \includegraphics[height=.4in]{pcross}}
                 \put(-35, 17){\fontsize{10}{11}$ a$}
  \put(0, 17){\fontsize{10}{11}$ b$}
   \put(-35, -20){\fontsize{10}{11}$ c$}
     \put(0, -20){\fontsize{11}{11}$ d$}
                     \end{picture} 
             \right> &-&
\left<\,\,
\begin{picture}(40,20)
      \raisebox{-13pt}{ \includegraphics[height=.4in]{ncross}}
                 \put(-35, 17){\fontsize{10}{11}$ a$}
  \put(0, 17){\fontsize{10}{11}$ b$}
   \put(-35, -20){\fontsize{10}{11}$ c$}
     \put(0, -20){\fontsize{11}{11}$ d$}
                     \end{picture} 
             \right> \\
             \\ \\
    &=& (q-q^{-1}) \left [ \left<\,\,
\begin{picture}(40,20)
      \raisebox{-13pt}{ \includegraphics[height=.4in]{ltsplit}}
       \put(-35, 17){\fontsize{10}{11}$ a$}
  \put(0, 17){\fontsize{10}{11}$ b$}
    \put(-35, -20){\fontsize{10}{11}$ c$}
     \put(0, -20){\fontsize{10}{11}$ d$}
      \end{picture}
 \right> + \left<\,\,
\begin{picture}(40,20)
      \raisebox{-13pt}{ \includegraphics[height=.4in]{gtsplit}}
       \put(-35, 17){\fontsize{10}{11}$ a$}
  \put(0, 17){\fontsize{10}{11}$ b$}
    \put(-35, -20){\fontsize{10}{11}$ c$}
     \put(0, -20){\fontsize{10}{11}$ d$}
      \end{picture}
 \right> +
 \left<\,\,
\begin{picture}(40,20)
      \raisebox{-13pt}{ \includegraphics[height=.4in]{eqsplit}}
       \put(-35, 17){\fontsize{10}{11}$ a$}
  \put(0, 17){\fontsize{10}{11}$ b$}
    \put(-35, -20){\fontsize{10}{11}$ c$}
     \put(0, -20){\fontsize{10}{11}$ d$}
      \end{picture}
 \right> \right ]\\ 
 \\ 
 &=& (q -q^{-1}) \left<\,\,
\begin{picture}(40,30)
      \raisebox{-13pt}{ \includegraphics[height=.4in]{2arcs}}
       \put(-35, 17){\fontsize{10}{11}$ a$}
  \put(0, 17){\fontsize{10}{11}$ b$}
    \put(-35, -20){\fontsize{10}{11}$ c$}
     \put(0, -20){\fontsize{10}{11}$ d$}
      \end{picture}
 \right>,\\
\end{eqnarray*}
\noindent 
where the last equality holds since the three states have non-zero evaluation if and only if $c=a$ and $d=b$, and since the spins $a$ and $b$ are either $a<b, a>b$ or $a=b$.

It follows that the polynomial $\brak{D}$ is an invariant of regular isotopy for oriented links. Moreover, the following hold:
\[ \left< \raisebox{-12pt}{ \includegraphics[height=.4in]{pcross}}\,\right> -  \left<\raisebox{-12pt}{ \includegraphics[height=.4in]{ncross}}\,\right>  = (q-q^{-1}) \left< \raisebox{-12pt}{  \includegraphics[height=.4in]{2arcs} }\, \right >\]

\[\left< \raisebox{-12pt}{  \includegraphics[height=.4in]{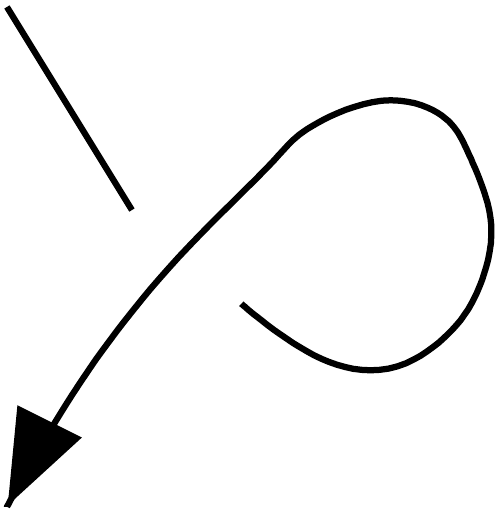}} \,\right> = q^{n}  \left<   \raisebox{-12pt}{  \includegraphics[height=.4in]{arc}}\,\right>,\,\,\,  \left< \raisebox{-12pt}{  \includegraphics[height=.4in]{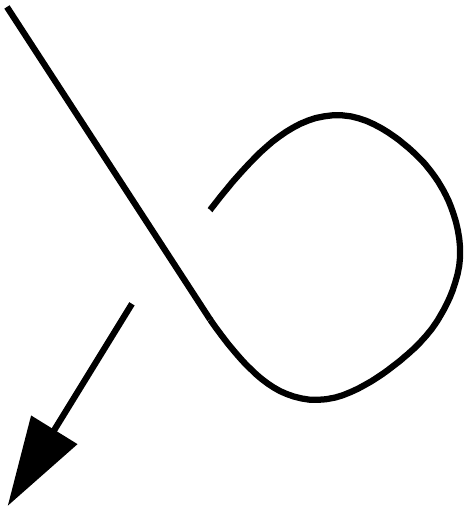}}\, \right> = q^{-n}  \left<   \raisebox{-12pt}{  \includegraphics[height=.4in]{arc}}\,\right>\]

\vspace{0.2cm}
\[\left< \raisebox{-9pt}{  \includegraphics[height=.3in]{pcirc}} \,\right>=  \left<\raisebox{-9pt}{  \includegraphics[height=.3in]{ncirc}} \,\right> = \displaystyle \frac{q^{n}-q^{-n}}{q-q^{-1}} = [n],\]
which implies that $\left<D\right>$ is the regular isotopy version of the $sl(n)$-link invariant.

\section{An invariant for singular links}\label{sec:inv-singlinks}

A \textit{singular link} is an immersion of a disjoint union of circles  in $\mathbb{R}^3$ which admits only finitely many singularities that are all transverse double points. 
A \textit{knotted graph} (also called a \textit{spacial graph}) is an embedding of a graph in $\mathbb{R}^3$. A singular link can be regarded as a 4-valent \textit{rigid-vertex embedding} of a graph in $\mathbb{R}^3$. In this paper we consider only 4-valent knotted graphs, that is graphs whose vertices have degree 4.

Two singular links are called \textit{equivalent} if their diagrams differ by a finite sequence of the classical Reidemeister moves together with the extended Reidemeister moves $R4$ and $R5$ shown in Figure~\ref{fig:extended-Reid-rigid}.

\begin{figure}[ht]
\[  \raisebox{-12pt}{\includegraphics[height=0.4in]{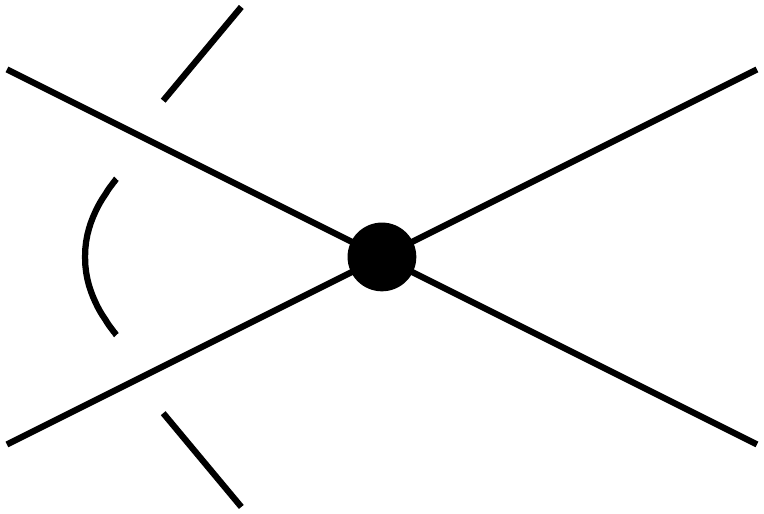}}\,\, \stackrel{R4}{\longleftrightarrow} \,\, \raisebox{-12pt}{\includegraphics[height=0.4in]{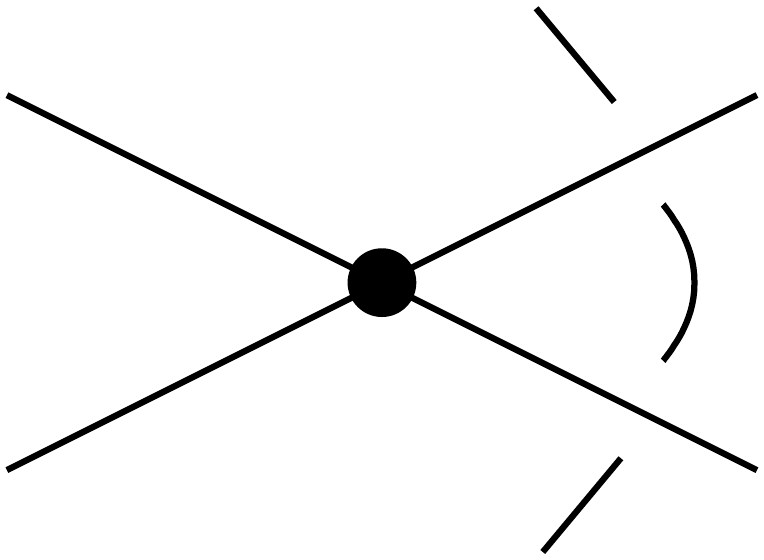}} \,\, \qquad \,\,  \raisebox{-12pt}{\includegraphics[height=0.4in]{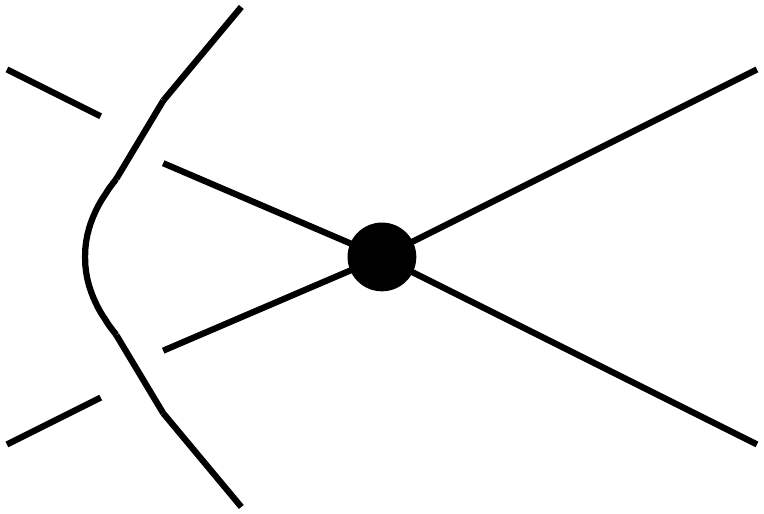}} \,\, \stackrel{R4}{\longleftrightarrow}\,\,  \raisebox{-12pt}{\includegraphics[height=0.4in]{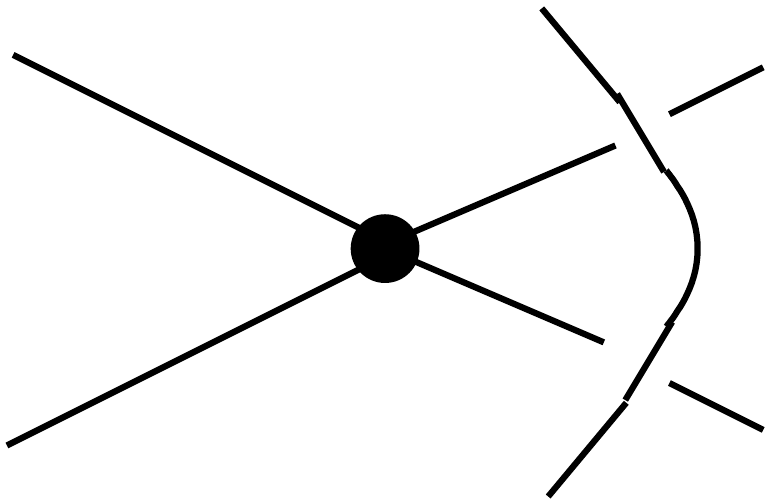}} \]\\
\[ \raisebox{-10pt}{\includegraphics[height=0.35in]{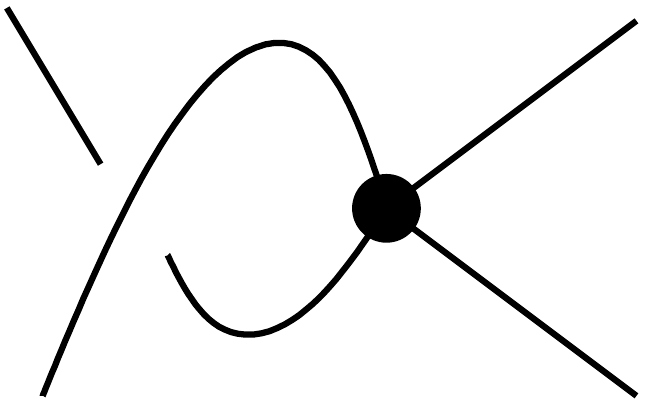}}\,\,\stackrel{R5}{ \longleftrightarrow} \reflectbox{\raisebox{-10pt}{\includegraphics[height=0.35in]{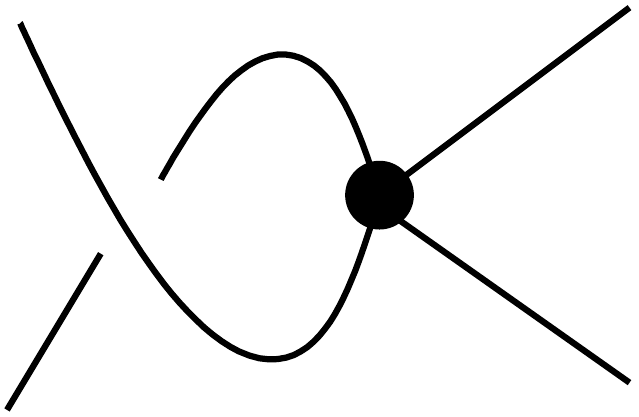}}}\]
\text{(rigid vertices)}
\caption{The moves $R4$ and $R5$}\label{fig:extended-Reid-rigid}
\end{figure}

Notice that the move $R5$ preserves the ordering of the edges meeting at a singular crossing. In graph-theoretical language this means that we regard a singular crossing as a rigid disk. Each disk has four arcs attached to it, and the cyclic order of these arcs is determined via the rigidity of the disk. A rigid vertex isotopy of the embeddings of such a graph  $G$ in three-space consists of affine motions of the disks, together with topological ambient isotopies of the edges of $G$. As mentioned above, the collection of moves that generate rigid vertex isotopy for diagrams of 4-valent graph embeddings are the classical Reidemeister moves coupled with the moves $R4$ and $R5$ depicted above (see~\cite{Ka0}). 

On the other hand, two 4-valent knotted graphs are equivalent if their diagrams differ by a finite sequence of the classical Reidemeister moves together with the extended Reidemeister moves $R4$ and $R6$. The Reidemeister move of type 6 is depicted in Figure~\ref{fig:non-rigid}. For more details on equivalent knotted graphs we refer the reader to Kauffman's work~\cite{Ka0}.
\begin{figure}[ht] 
\[ \raisebox{-10pt}{\includegraphics[height=0.35in]{N5-left1}}\,\,\stackrel{R6}{ \longleftrightarrow} \,\, \raisebox{-10pt}{\includegraphics[height=0.35in]{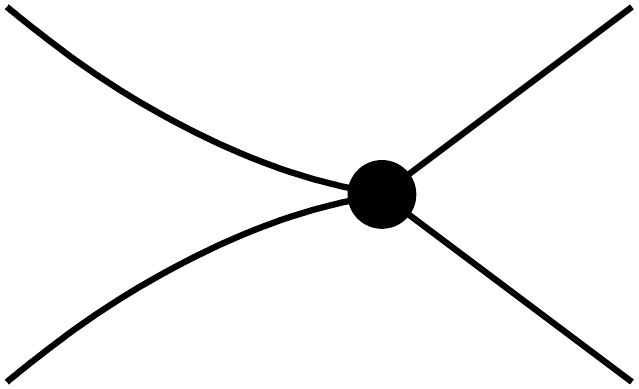}} \,\, \stackrel{R6}{\longleftrightarrow} \,\,  \raisebox{-10pt}{\includegraphics[height=0.35in]{N5-left2}}\]
 \text{(non-rigid vertices)}
\caption{The move $R6$}\label{fig:non-rigid}
\end{figure}

In this paper all singular links and knotted graphs are oriented. Our first goal is to extend the Yang-Baxter state model for the $sl(n)$ link polynomial described in Section~\ref{sec:YB-sl(n)poly} to oriented singular links. 

Given a singular link diagram $G$, we label its edges with spins from the equally spaced index set $I_n=\{ 1-n,3-n, \dots,n-3,n-1\}$, for $n \in \mathbb{Z}, n \geq 2$ and we decompose the classical crossings according to the skein relations in Figure~\ref{fig:crossings}. We need to define a skein relation involving a singular crossing of $G$. For example, we can impose the following skein relation, for some $\alpha, \beta \in \mathbb{Z}[q, q^{-1}]$:
\vspace{0.1cm}
\begin{eqnarray}\label{eq:general-sing-crossing}
\left<\,\,
\begin{picture}(40,20)
      \raisebox{-13pt}{ \includegraphics[height=.4in]{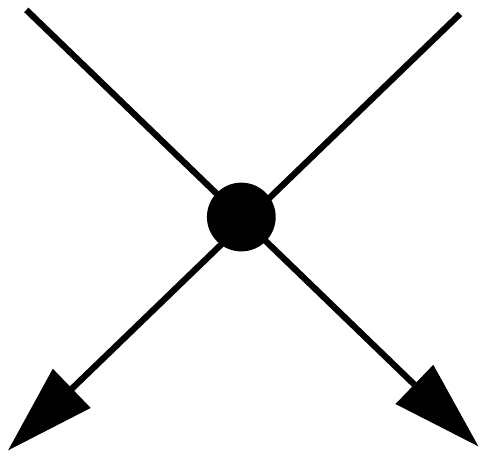}}
                \put(-37, 17){\fontsize{10}{11}$ a$}
  \put(0, 17){\fontsize{10}{11}$ b$}
    \put(-37, -20){\fontsize{10}{11}$ c$}
     \put(0, -20){\fontsize{10}{11}$ d$}
                                          \end{picture} 
               \right>
               = \alpha \left<\,\,
\begin{picture}(40,20)
      \raisebox{-13pt}{ \includegraphics[height=.4in]{pcross}}
 \put(-37, 17){\fontsize{10}{11}$ a$}
  \put(0, 17){\fontsize{10}{11}$ b$}
    \put(-37, -20){\fontsize{10}{11}$ c$}
     \put(0, -20){\fontsize{10}{11}$ d$}
      \end{picture}
     \right>
      +\beta \left<\,\,
\begin{picture}(40,20)
      \raisebox{-13pt}{ \includegraphics[height=.4in]{ncross}}
       \put(-37, 17){\fontsize{10}{11}$ a$}
  \put(0, 17){\fontsize{10}{11}$ b$}
    \put(-37, -20){\fontsize{10}{11}$ c$}
     \put(0, -20){\fontsize{10}{11}$ d$}
                    \end{picture}  \right>           
 \end{eqnarray}
\vspace{0.5cm}

Then we evaluate the resulting states of $G$ using the formula \eqref{eq:state-eval}.              
Putting all together, we obtain a Laurent polynomial $\brak{G}$ associated with a singular link diagram $G$, given by
\begin{eqnarray}\label{eq:sing-polyn-eval}
\left< G \right>= \sum_{\sigma} b_{\sigma}\, \brak{\sigma} = \sum_{\sigma} b_{\sigma}\, q^{||\sigma||}, 
\end{eqnarray}
where the sum is taken over all states $\sigma$ of $G$ and where $b_{\sigma}$ is the product of the weights associated with a state $\sigma$ according to the skein relations given in Equation~\eqref{eq:general-sing-crossing} and Figure~\ref{fig:crossings}.

We remind the reader that given an invariant of regular isotopy for classical links, it can be extended via the relation~\eqref{eq:general-sing-crossing} to a regular isotopy invariant of singular links. Translating this into our case, we arrive at the following result:

\begin{theorem}
The Laurent polynomial $\brak{G}(q) \in \Z[q, q^{-1}]$ is an invariant of regular isotopy for oriented singular links $G$, for any $\alpha, \beta \in \Z[q, q^{-1}]$, and satisfies
\[ \left< \raisebox{-12pt}{ \includegraphics[height=.4in]{pcross}}\,\right> -  \left<\raisebox{-12pt}{ \includegraphics[height=.4in]{ncross}}\,\right>  = (q-q^{-1}) \left< \raisebox{-12pt}{  \includegraphics[height=.4in]{2arcs} }\, \right >\]
\[\left< \raisebox{-12pt}{  \includegraphics[height=.4in]{poskink}} \,\right> = q^{n}  \left<   \raisebox{-12pt}{  \includegraphics[height=.4in]{arc}}\,\right>,\,\,\,  \left< \raisebox{-12pt}{  \includegraphics[height=.4in]{negkink}}\, \right> = q^{-n}  \left<   \raisebox{-12pt}{  \includegraphics[height=.4in]{arc}}\,\right>\]
\[\left< \raisebox{-9pt}{  \includegraphics[height=.3in]{pcirc}} \,\right>=  \left<\raisebox{-9pt}{  \includegraphics[height=.3in]{ncirc}} \,\right>  = [n].\]
\end{theorem}
\begin{proof}
Since $\brak{G}$ is an extension of the Yang-Baxter state model for the $sl(n)$ link invariant, it follows at once that $\brak{G}$ is invariant under the type 2 and type 3 Reidemeister moves, and that it satisfies the above three relations.  It is easy to see that $\brak{G}$ is invariant under the extended $R4$ move; this follows from Equation~\eqref{eq:general-sing-crossing} and the fact that $\brak {G}$ is invariant under the Reidemeister move of type 3. We show below that $\brak{G}$ is invariant under the move $R5$:

\begin{eqnarray*}
 \left<\,
       \raisebox{-15pt}{\reflectbox{ \includegraphics[height=0.45in]{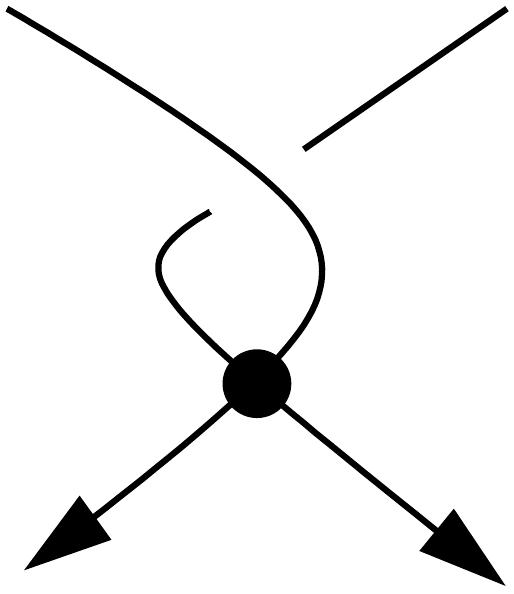}}}
 \right>
 &=& \alpha
 \left<
      \raisebox{-13pt}{ \includegraphics[height=.4in]{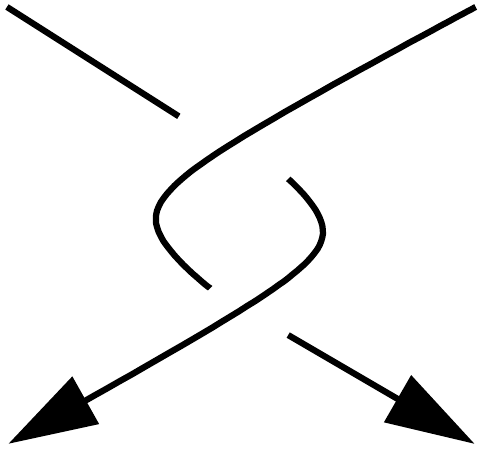}}
 \right>
 +  \beta
  \left<
      \raisebox{-13pt}{ \includegraphics[height=.4in]{reid2a}}
 \right> 
 \\
 &\stackrel{R2}{=}& \alpha
 \left<
      \raisebox{-13pt}{ \includegraphics[height=.4in]{twist}}
 \right>
 +  \beta
  \left<
      \reflectbox{\raisebox{-13pt}{ \includegraphics[height=.4in]{reid2a}}}
 \right>
\\
 &=& \left<
      \raisebox{-15pt}{ \includegraphics[height=.45in]{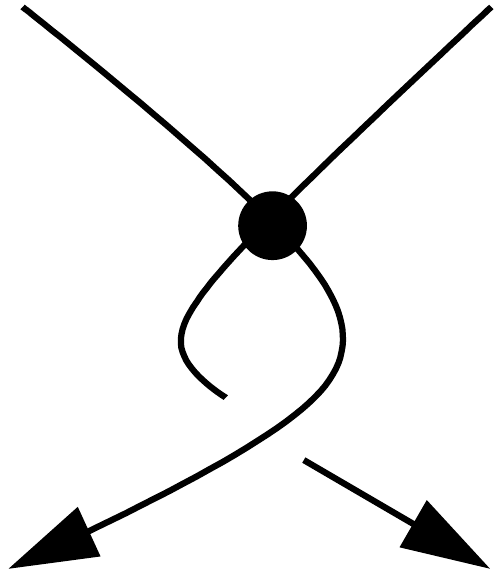}}\,
 \right>
\end{eqnarray*}
which completes the proof. 
\end{proof}

For the remaining of the paper, we work with $\alpha = \displaystyle \frac{q}{q-q^{-1}}$ and $\beta = \displaystyle \frac{-q^{-1}}{q-q^{-1}}$ in Equation~\eqref{eq:general-sing-crossing}. This results in the singular crossing decomposition displayed in Figure~\ref{fig:sing-crossing}.
\begin{figure}[ht] 
\[
\left<\,\,
\begin{picture}(40,20)
      \raisebox{-13pt}{ \includegraphics[height=.4in]{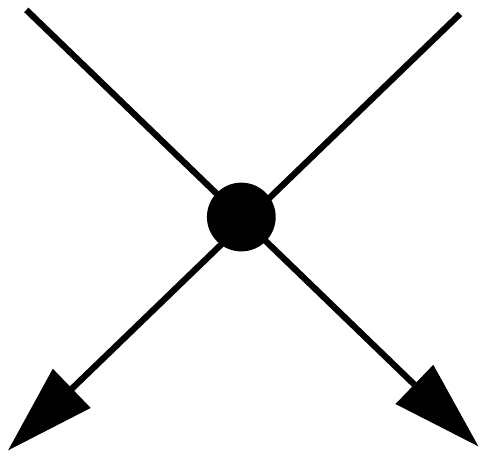}}
                \put(-37, 17){\fontsize{10}{11}$ a$}
  \put(0, 17){\fontsize{10}{12}$ b$}
    \put(-37, -20){\fontsize{10}{12}$ c$}
     \put(0, -20){\fontsize{10}{12}$ d$}
                \end{picture} 
              \right>
               = q\left<\,\,
\begin{picture}(40,20)
      \raisebox{-13pt}{ \includegraphics[height=.4in]{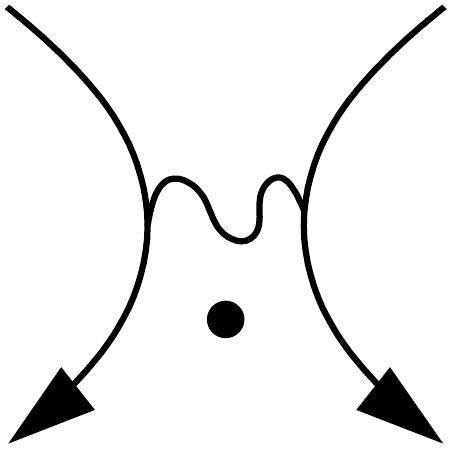}}
     \put(-37, 17){\fontsize{10}{11}$ a$}
  \put(0, 17){\fontsize{10}{11}$ b$}
    \put(-37, -20){\fontsize{10}{11}$ c$}
     \put(0, -20){\fontsize{10}{11}$ d$}
      \end{picture}
 \right>
       +q^{-1} \left<\,\,
\begin{picture}(40,20)
      \raisebox{-13pt}{ \includegraphics[height=.4in]{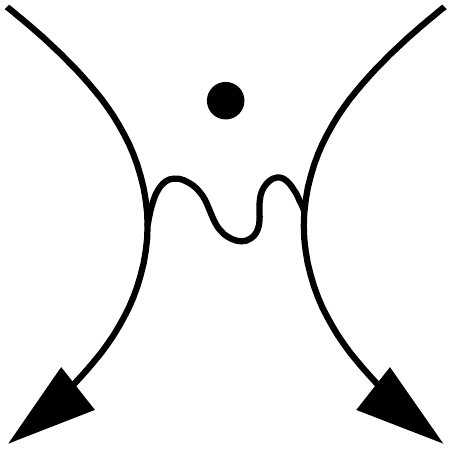}}
     \put(-37, 17){\fontsize{10}{11}$ a$}
  \put(0, 17){\fontsize{10}{11}$ b$}
    \put(-37, -20){\fontsize{10}{11}$ c$}
     \put(0, -20){\fontsize{10}{11}$ d$}
      \end{picture}  
       \right>
      +(q+q^{-1}) \left<\,\,
\begin{picture}(40,20)
      \raisebox{-13pt}{ \includegraphics[height=.4in]{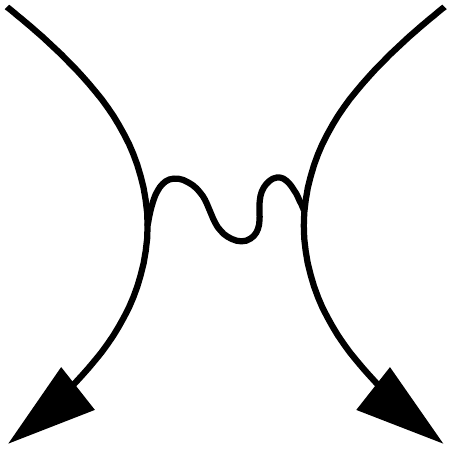}}
   \put(-37, 17){\fontsize{10}{11}$ a$}
  \put(0, 17){\fontsize{10}{11}$ b$}
    \put(-37, -20){\fontsize{10}{11}$ c$}
     \put(0, -20){\fontsize{10}{11}$ d$}
                    \end{picture}  
\right>
      + \left<\,\,
\begin{picture}(40,20)
     \raisebox{-13pt}{  \includegraphics[height=.4in]{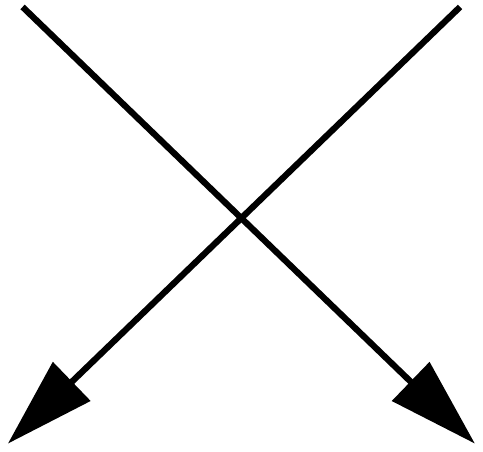}}
 \put(-37, 17){\fontsize{10}{11}$ a$}
  \put(0, 17){\fontsize{10}{11}$ b$}
    \put(-37, -20){\fontsize{10}{11}$ c$}
     \put(0, -20){\fontsize{10}{11}$ d$}
              \put(-18, 7){\fontsize{10}{11}$ \neq$}
      \end{picture} 
 \right> 
\]
\caption{Singular crossing decomposition} \label{fig:sing-crossing}
\end{figure}

Notice that the evaluation of a singular crossing is non-zero only when the spins $a, b, c$ and $d$ associated with the four edges incident with the singular crossing satisfy the conservation law $a+b =c+d$. Specifically, the evaluation of a singular crossing is non-zero only when $a=c$ and $b = d$ or $d=a \neq b=c$. 
It is important to note the difference between the left-hand side of the skein relation in Figure ~\ref{fig:sing-crossing} and the last term in the right-hand side of the same skein relation: the latter makes use of spins $a, b, c, d$ such that $d=a \neq b=c$, and is a decorated state of singular crossings.
  
Denote by $Q$ the $n^2 \times n^2$ square matrix corresponding to a singular crossings. Then the latter skein relation can be rewritten in terms of the entries of the matrix $Q$ as follows:
\[Q^{ab}_{cd} = q [a<b]  \delta^a_c \delta^b_d  + q^{-1} [a>b]  \delta^a_c \delta^b_d  +(q+q^{-1}) [a=b]   \delta^a_c \delta^b_d +[a\neq b]   \delta^a_d \delta^b_c,\]
for all $a, b, c, d \in I_n$. Note that since $\brak{G}$ is invariant under the move $R5$, it implies that $RQ = QR$ and $\overline{R}Q = Q\overline{R}$. 

\begin{example}
Using abstract tensor diagrams and matrices $R, \overline{R}$ and $Q$, the Laurent polynomial $\brak{G}$ associated with the diagram $G$ depicted in Figure~\ref{fig:tensor-singlink} is given by the following expression:
\[ \brak{D} = \sum_{a, b, \dots, n \in  I_n} \overleftarrow{M}_{ad} \overleftarrow{M}_{bc}R^{ab}_{ef} R^{ef}_{ij} \overrightarrow{M}^{im} Q^{mj}_{nk} \overrightarrow{M}^{nl} R^{lk}_{hg} Q^{hg}_{dc} \]
where the sum is over all possible choices of indices (spins from $I_n$) in the expression.
\end{example}

\begin{figure}[ht]
\includegraphics[height =1.3in, width = 1in]{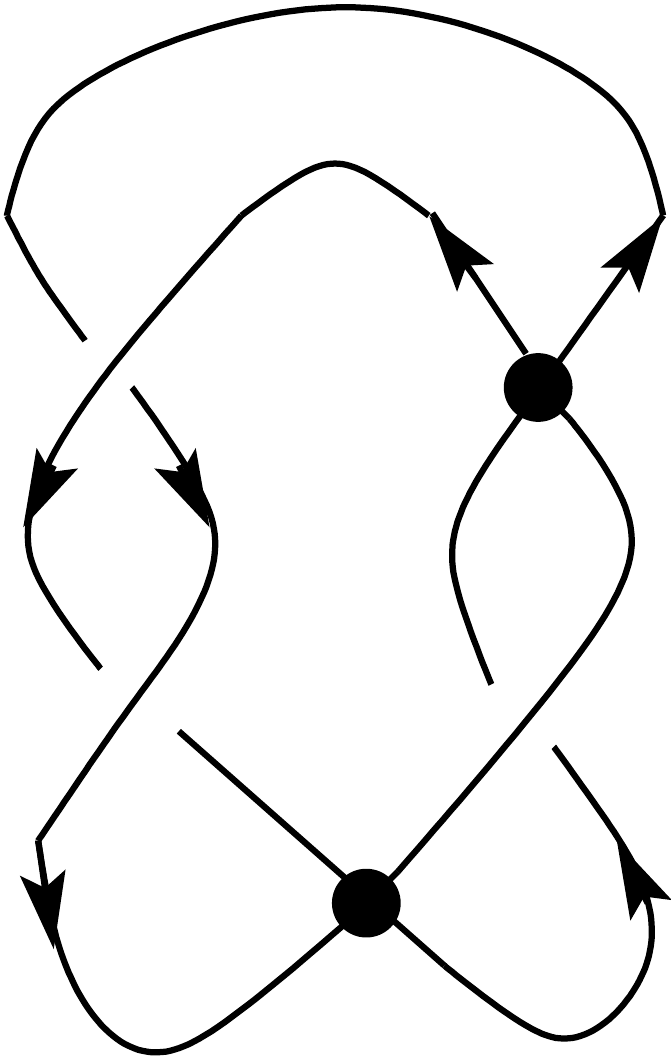}
  \put(-75,67){\fontsize{10}{11}$ a$}
  \put(-48, 66){\fontsize{10}{11}$ b$}
  \put(-30, 66){\fontsize{10}{11}$ c$}
  \put(-3, 66){\fontsize{10}{11}$ d$}
   \put(-77, 43){\fontsize{10}{11}$ e$}
  \put(-48, 43){\fontsize{10}{11}$ f$}
  \put(-30, 44){\fontsize{10}{11}$ g$}
  \put(-3, 43){\fontsize{10}{11}$ h$}
   \put(-76, 20){\fontsize{10}{11}$ i$}
  \put(-51, 17){\fontsize{10}{11}$ j$}
  \put(-31, 22){\fontsize{10}{11}$ k$}
  \put(-1, 20){\fontsize{10}{11}$ l$}
   \put(-46, 0){\fontsize{10}{11}$ m$}
  \put(-28, 0){\fontsize{10}{11}$ n$}
\caption{An abstract tensor singular link diagram}\label{fig:tensor-singlink}
\end{figure}
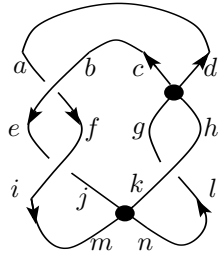

\begin{proposition} \label{prop:poly-r6} 
For $\alpha = \displaystyle \frac{q}{q-q^{-1}}$ and $\beta = \displaystyle \frac{-q^{-1}}{q-q^{-1}}$, the following skein relations hold:
\begin{eqnarray*}
\left< \raisebox{-15pt}{ \reflectbox{\includegraphics[height=0.5in]{r61}}}\, \right>
               &=& q \,\left<  \raisebox{-13pt}{ \includegraphics[height=.4in]{flatcross}}\, \right>\\
\left<  \raisebox{-15pt}{ \includegraphics[height=0.5in]{r61}}\, \right>
               &=& q^{-1}\left< \raisebox{-13pt}{ \includegraphics[height=.4in]{flatcross}} \, \right>.
\end{eqnarray*}
\end{proposition}
\begin{proof}
The first step below makes use of the crossing decomposition from Figure~\ref{fig:crossings} applied to the left-hand side of the first equality, which results in three diagrams. Then, we apply the skein relation depicted in Figure~\ref{fig:sing-crossing} to each of the resulting diagrams.
\begin{eqnarray*}
&&\left< \raisebox{-15pt}{ \reflectbox{\includegraphics[height=0.5in]{r61}}}\, \right>\\
& &= (q-q^{-1})\left< \raisebox{-15pt}{ \reflectbox{\includegraphics[height=0.5in]{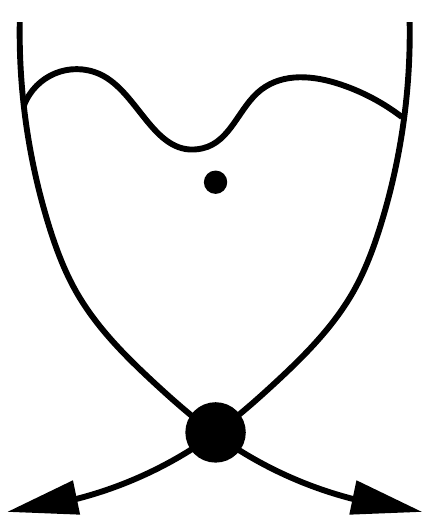}}}\, \right>
+q\left< \raisebox{-15pt}{ \reflectbox{\includegraphics[height=0.5in]{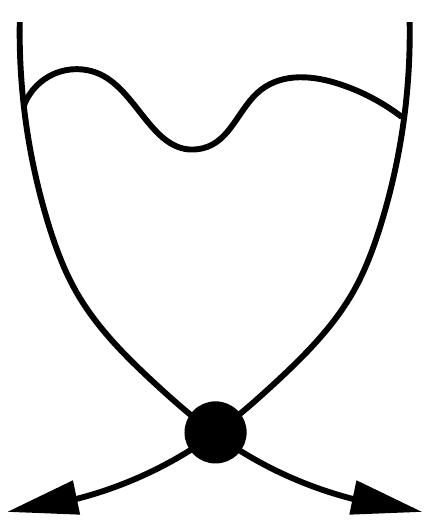}}}\, \right>
+\left< \raisebox{-15pt}{ \reflectbox{\includegraphics[height=0.5in]{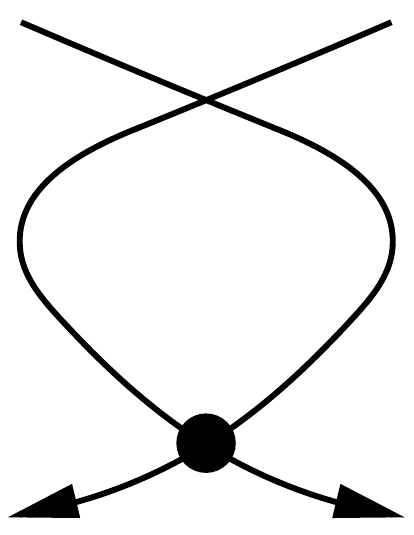}}}\put(-18, 17){\fontsize{10}{11}$ \neq$}\, \right>\\
& & = (q-q^{-1}) \left[q\left< \raisebox{-15pt}{ \reflectbox{\includegraphics[height=0.5in]{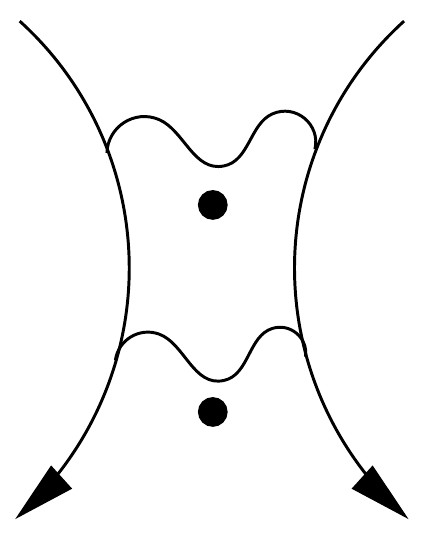}}}\, \right> +
q^{-1} \left< \raisebox{-15pt}{ \reflectbox{\includegraphics[height=0.5in]{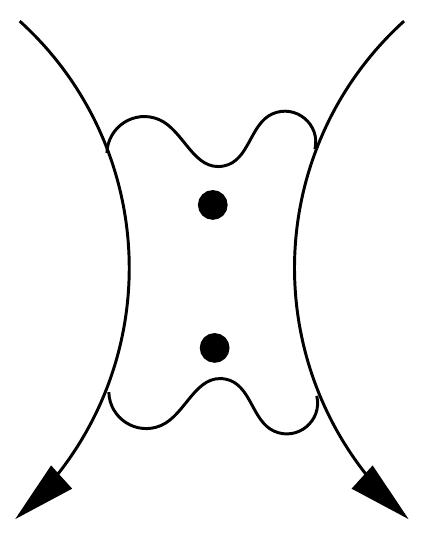}}}\, \right> +
(q+q^{-1})\left< \raisebox{-15pt}{ \reflectbox{\includegraphics[height=0.5in]{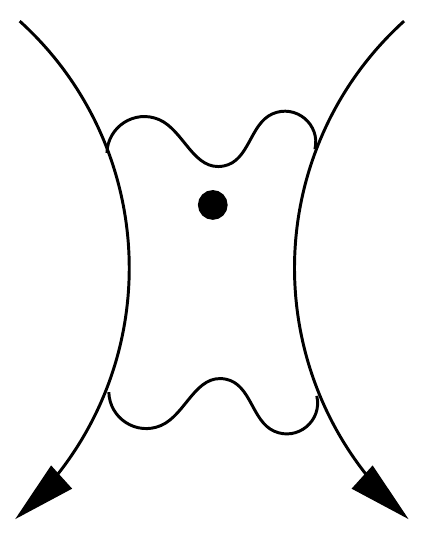}}}\, \right> +
\left< \raisebox{-15pt}{ \reflectbox{\includegraphics[height=0.5in]{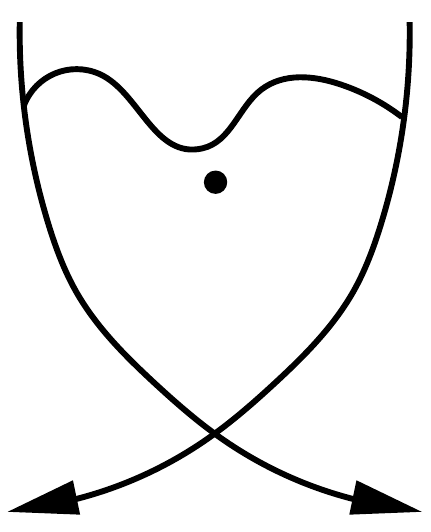}}}\put(-18, -3){\fontsize{10}{11}$ \neq$}\, \right> \right]\\
&&\hspace{0.4cm}+ q \left[q\left< \raisebox{-15pt}{ \reflectbox{\includegraphics[height=0.5in]{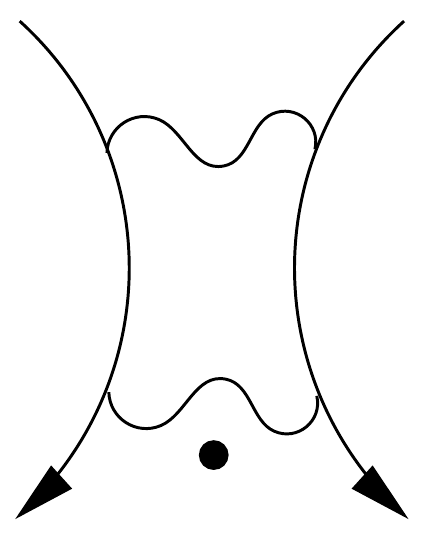}}}\, \right> +
q^{-1}\left< \raisebox{-15pt}{ \reflectbox{\includegraphics[height=0.5in]{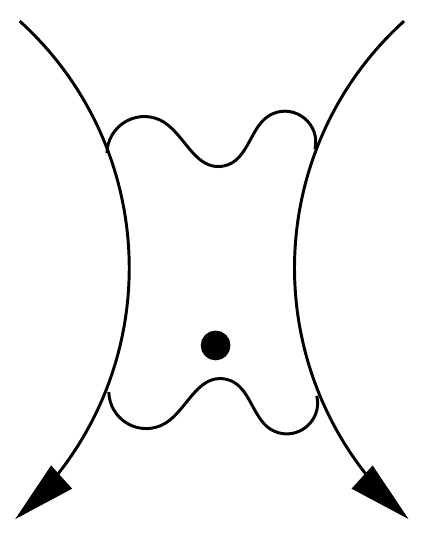}}}\, \right> +
(q+q^{-1})\left< \raisebox{-15pt}{ \reflectbox{\includegraphics[height=0.5in]{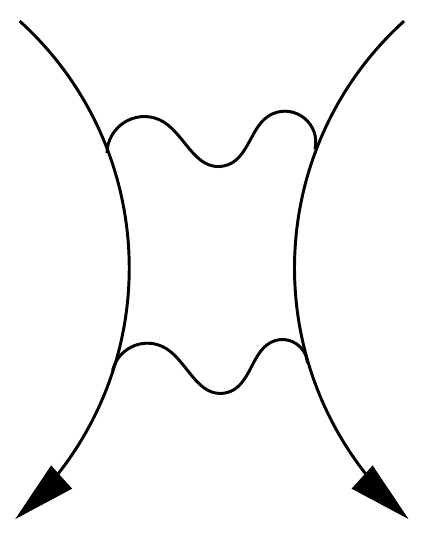}}}\, \right> +
\left< \raisebox{-15pt}{ \reflectbox{\includegraphics[height=0.5in]{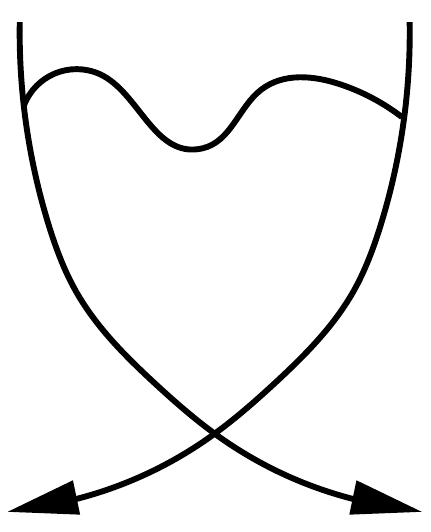}}}\put(-18, -3){\fontsize{10}{11}$ \neq$}\, \right>\right]\\
&&\hspace{0.4cm} + \left[q\left< \raisebox{-15pt}{ \reflectbox{\includegraphics[height=0.5in]{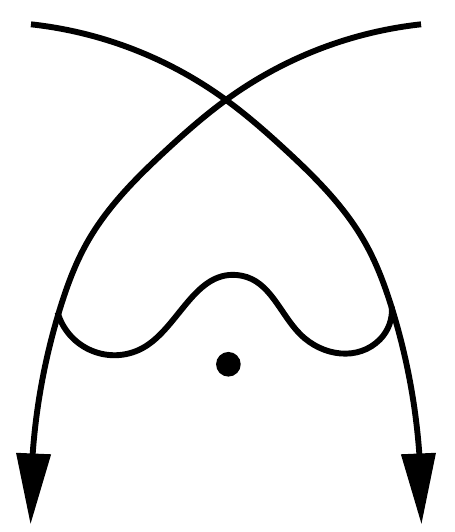}}}\put(-18, 17){\fontsize{10}{11}$ \neq$}\, \right> +
q^{-1}\left< \raisebox{-15pt}{ \reflectbox{\includegraphics[height=0.5in]{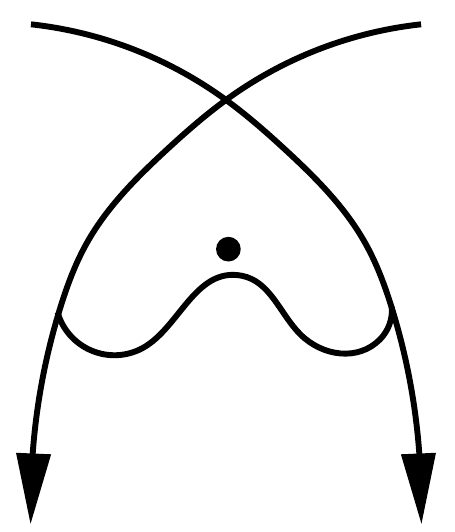}}}\put(-18, 17){\fontsize{10}{11}$ \neq$}\, \right> +
(q+q^{-1})\left< \raisebox{-15pt}{ \reflectbox{\includegraphics[height=0.5in]{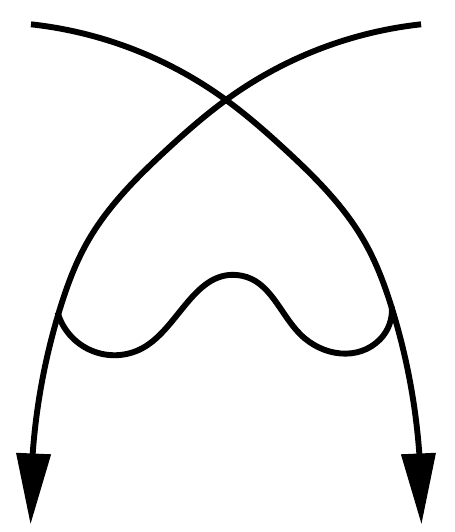}}}\put(-18, 17){\fontsize{10}{11}$ \neq$}\, \right> +
\left< \raisebox{-15pt}{ \reflectbox{\includegraphics[height=0.5in]{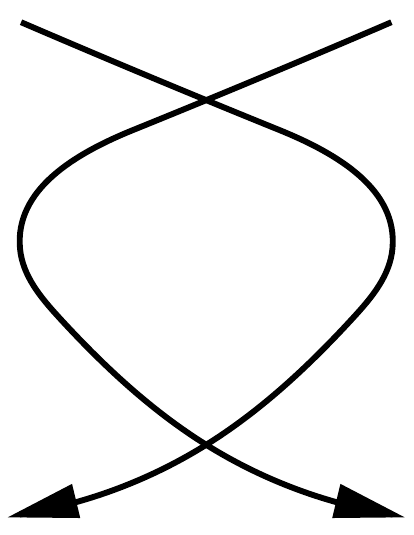}}}\put(-18, 17){\fontsize{10}{11}$ \neq$}\put(-18, -3){\fontsize{10}{11}$ \neq$}\, \right>\right]\\
\end{eqnarray*}
Some of the diagrams above evaluate to 0, due to incompatible labelings of the strands. Specifically, 
\[
\left< \raisebox{-15pt}{ \reflectbox{\includegraphics[height=0.5in]{prop2_5}}}\, \right>  = 
\left< \raisebox{-15pt}{ \reflectbox{\includegraphics[height=0.5in]{prop2_6}}}\, \right> = 
\left< \raisebox{-15pt}{ \reflectbox{\includegraphics[height=0.5in]{prop2_8}}}\, \right> =
\left< \raisebox{-15pt}{ \reflectbox{\includegraphics[height=0.5in]{prop2_9}}}\, \right> = 0,
\]
\[
\left< \raisebox{-15pt}{ \reflectbox{\includegraphics[height=0.5in]{prop2_11}}}\put(-18, -3){\fontsize{10}{11}$ \neq$}\, \right> =
\left< \raisebox{-15pt}{ \reflectbox{\includegraphics[height=0.5in]{prop2_14}}}\put(-18, 17){\fontsize{10}{11}$ \neq$}\, \right> =
0.
\]
In addition, note that
\[
\left< \raisebox{-15pt}{ \reflectbox{\includegraphics[height=0.5in]{prop2_7}}}\put(-18, -3){\fontsize{10}{11}$ \neq$}\, \right> =
\left< \raisebox{-15pt}{ \reflectbox{\includegraphics[height=0.5in]{prop2_13}}}\put(-18, 17){\fontsize{10}{11}$ \neq$}\, \right>,
\hspace{1cm}
\left< \raisebox{-15pt}{ \reflectbox{\includegraphics[height=0.5in]{prop2_12}}}\put(-18, 17){\fontsize{10}{11}$ \neq$}\, \right> +
\left< \raisebox{-15pt}{ \reflectbox{\includegraphics[height=0.5in]{prop2_13}}}\put(-18, 17){\fontsize{10}{11}$ \neq$}\, \right> =
\left< \raisebox{-13pt}{  \includegraphics[height=.4in]{flat}}
              \put(-18, 7){\fontsize{10}{11}$ \neq$}\, \right>.
\]
Putting these together, we have
\begin{eqnarray*}
\left< \raisebox{-15pt}{ \reflectbox{\includegraphics[height=0.5in]{r61}}}\, \right> & = &
q (q-q^{-1}) \left< \raisebox{-15pt}{ \reflectbox{\includegraphics[height=0.45in]{ltsplit}}}\, \right> +
q(q+q^{-1})\left< \raisebox{-15pt}{ \reflectbox{\includegraphics[height=0.45in]{eqsplit}}}\, \right> \\
&&+q\left<\raisebox{-13pt}{  \includegraphics[height=.45in]{flat}}\put(-20, 9){\fontsize{10}{11}$ \neq$}\,\right> +
\left< \raisebox{-15pt}{ \reflectbox{\includegraphics[height=0.5in]{prop2_15}}}\put(-18, 17){\fontsize{10}{11}$ \neq$}\put(-18, -3){\fontsize{10}{11}$ \neq$}\, \right>.
\end{eqnarray*}
Moreover, since 
\[  
\left< \raisebox{-15pt}{ \reflectbox{\includegraphics[height=0.5in]{prop2_15}}}\put(-18, 17){\fontsize{10}{11}$ \neq$}\put(-18, -3){\fontsize{10}{11}$ \neq$}\, \right> = \left< \raisebox{-15pt}{ \reflectbox{\includegraphics[height=0.45in]{ltsplit}}}\, \right> +\left< \raisebox{-15pt}{ \reflectbox{\includegraphics[height=0.45in]{gtsplit}}}\, \right>,
\]
the previous equality is equivalent to
\begin{eqnarray*}
\left< \raisebox{-15pt}{ \reflectbox{\includegraphics[height=0.5in]{r61}}}\, \right> & = &(q^2-1) \left< \raisebox{-15pt}{ \reflectbox{\includegraphics[height=0.45in]{ltsplit}}}\, \right> +q(q+q^{-1})\left< \raisebox{-15pt}{ \reflectbox{\includegraphics[height=0.45in]{eqsplit}}}\, \right>\\
&&+q\left<\raisebox{-13pt}{  \includegraphics[height=.45in]{flat}}\put(-20, 9){\fontsize{10}{11}$ \neq$}\,\right> + \left< \raisebox{-15pt}{ \reflectbox{\includegraphics[height=0.45in]{ltsplit}}}\, \right> +\left< \raisebox{-15pt}{ \reflectbox{\includegraphics[height=0.45in]{gtsplit}}}\, \right>\\
&=&  q\left< \raisebox{-13pt}{ \includegraphics[height=0.4in]{flatcross}}\, \right>.
\end{eqnarray*}
Therefore, the first skein relation holds. The second relation is proved in a similar fashion. \end{proof}

The \textit{mirror image} of a singular link with diagram $G$ is the singular link whose diagram $G^*$ is obtained from $G$ by replacing each (classical) positive crossing with a negative crossing and vice versa. A singular link is said to be \textit{achiral} if it is ambient isotopic to its mirror image. Otherwise, $G$ is called \textit{chiral}.

\begin{proposition}
Let $G$ be an oriented singular link and $G^*$ its mirror image. Then the polynomial $\brak{G^*}$ is obtained from $\brak{G}$ by replacing $q$ with $q^{-1}$. That is,
\[\brak{G^*}(q) = \brak{G}(q^{-1}).\]
\end{proposition}
\begin{proof}
$G^*$ is obtained from $G$ by reversing all classical crossings, which has the effect of interchanging $q$ and $q^{-1}$ in the definition of $\brak{\,\, \cdot \,\,}$. On the other hand, the evaluation of a singular crossing remains the same when $q$ and $q^{-1}$ are interchanged. Therefore, the statement holds.
\end{proof}
\begin{corollary}
If $\brak{G}(q) \neq \brak{G}(q^{-1})$, then $G$ is a chiral singular link.
\end{corollary}
\begin{proposition}
Let $G_1 \cup G_2$ be the disjoint union of oriented singular links $G_1$ and $G_2$. Then,
\[  \brak{G_1 \cup G_2} = \brak{G_1} \brak{G_2}. \]
\end{proposition}
\begin{proof}
Notice that this formula holds when $G_1$ and $G_2$ are classical links (that is, when $G_1$ and $G_2$ have no singular crossings). Then the statement is verified for singular links using a standard proof by induction on the number of singular crossings, and thus it is omitted.
\end{proof}
A singular link diagram $G$ is a \textit{connected sum}, denoted by $G = G_1 \# G_2$, if it is displayed as two disjoint singular link diagrams $G_1$ and $G_2$ connected by parallel embedded arcs, up to planar isotopy, as in Figure~\ref{fig:connected-sum}. The following result holds for classical links, and can be proved for singular links, as well, by induction on the number of singular crossings.
\begin{figure}[ht]
\includegraphics[height=0.7in]{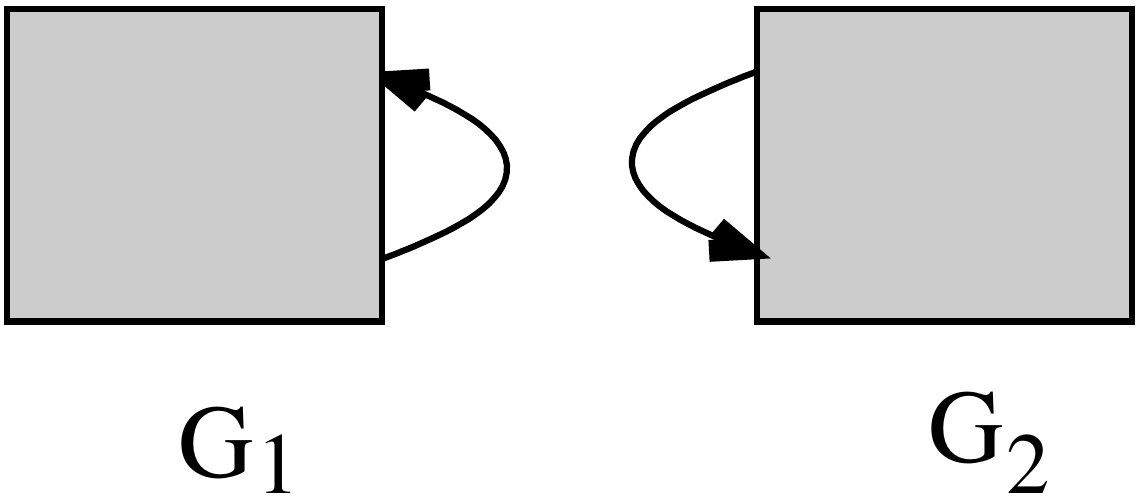} \hspace{1cm} \includegraphics[height=0.7in]{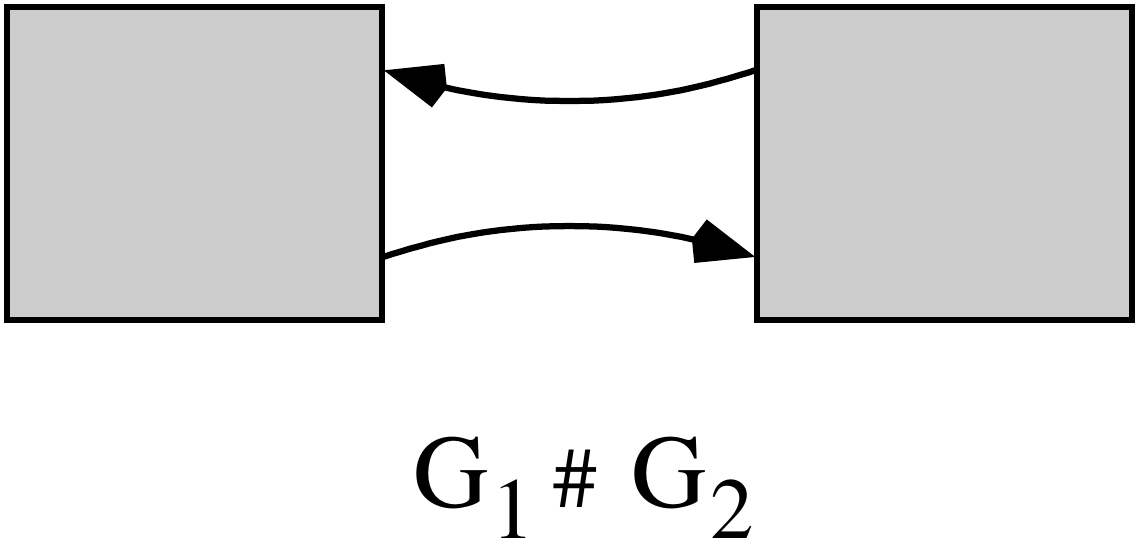}
\caption{A connected sum} \label{fig:connected-sum}
\end{figure}

\begin{proposition}
Let $G$ be an oriented singular link diagram with the property that $G = G_1 \# G_2$, for some oriented singular link diagrams $G_1$ and $G_2$. Then the polynomial $\brak{G}$ can be computed as follows:
\[ \brak{G} = \frac{1}{[n]}\brak{G_1} \brak{G_2}.  \]
\end{proposition}

\section{Representations of the singular braid monoid}\label{sec:repres}

Let $G$ be a singular link diagram and consider the polynomial $\brak{G}$ defined by the Equation~\eqref{eq:sing-polyn-eval}, with $\alpha = \displaystyle \frac{q}{q-q^{-1}}$ and $\beta = \displaystyle \frac{-q^{-1}}{q-q^{-1}}$. In this section we show how to use the Yang-Baxter state model for $\brak{G}$ to define, for each $n \in \Z , n\geq 2$, a representation of the singular braid monoid into a matrix algebra.

Recall the $n^2 \times n^2$ matrices $R$ and $\overline{R}$ associated with a positive and a negative crossing, respectively (and satisfying the YBE), and the $n^2 \times n^2$ matrix $Q$ corresponding to a singular crossing. These matrices have entries given by:
\begin{eqnarray*}
R^{ab}_{cd}&=&(q-q^{-1}) [a<b]  \text{ } \delta^a_c \text{ }\delta^b_d  + q [a=b] \text{ }  \delta^a_c  \text{ } \delta^b_d  +[a\neq b]  \text{ }  \delta^a_d  \text{ } \delta^b_c\\
\overline{R}^{ab}_{cd}&=&(q^{-1}-q) [a>b]  \text{ } \delta^a_c  \text{ } \delta^b_d  + q^{-1} [a=b]  \text{ } \delta^a_c  \text{ } \delta^b_d  +[a\neq b]   \text{ } \delta^a_d  \text{ } \delta^b_c \\
Q^{ab}_{cd} &=& q [a<b]  \delta^a_c \delta^b_d  + q^{-1} [a>b]  \delta^a_c \delta^b_d  +(q+q^{-1}) [a=b]   \delta^a_c \delta^b_d +[a\neq b]   \delta^a_d \delta^b_c
\end{eqnarray*}
for all $a, b, c, d \in I_n$.
That is, the matrices $R = (R^{ab}_{cd})$ and $\overline{R} = (\overline{R}^{ab}_{cd})$ look as follows:
\[
R^{ab}_{cd} = \begin{cases} q-q^{-1}\,\,\, \text{ if} \,\,\,\, c =a < b =d\\
q \hspace{1.2 cm}\text{ if} \,\,\,\,c=a=b=d\\
1\hspace{1.2cm} \text{ if }\,\,d =a \neq b =c \\
0 \hspace{1cm} \text{otherwise}
\end{cases}
\hspace{0.5cm}
\overline{R}^{ab}_{cd} = \begin{cases} q^{-1} - q\,\,\, \text{ if} \,\,\,\, c =a > b =d\\
q^{-1} \hspace{0.8cm}\text{ if} \,\,\,\,c=a=b=d\\
1\hspace{1.2cm} \text{ if }\,\, d =a \neq b =c \\
0 \hspace{1cm} \text{otherwise}
\end{cases}
\]
In addition, the matrix $Q = (Q^{ab}_{cd})$ is given by:
\[ 
Q^{ab}_{cd} = \begin{cases} q+q^{-1}\,\,\, \text{ if} \,\,\,\, c =a = b =d\\
q \hspace{1.2 cm}\text{ if} \,\,\,\,c=a < b=d\\
q^{-1} \hspace{0.8 cm}\text{ if} \,\,\,\,c=a > b=d\\
1\hspace{1.2cm} \text{ if }\,\,d =a \neq b =c \\
0 \hspace{1cm} \text{otherwise}
\end{cases}
\]
For $n =2$ the index set is $I_2=\{-1, 1\}$, giving the following matrices:
\[  R_2=\begin{bmatrix}
 q&0&0&0\\ 0&q-q^{-1}&1&0\\ 0&1&0&0\\ 0&0&0&q \end{bmatrix}, \hspace{1cm} \overline{R}_2=\begin{bmatrix}
 q^{-1}&0&0&0\\ 0&0&1&0\\ 0&1&q^{-1}-q&0\\ 0&0&0&q^{-1} \end{bmatrix}\] 

 \[  Q_2=\begin{bmatrix}
 q+q^{-1}&0&0&0\\ 0&q&1&0\\ 0&1&q^{-1}&0\\ 0&0&0&q+q^{-1} \end{bmatrix}\] 
 For $n=3$ the index set becomes $I_3=\{-2,0,2\}$ and the corresponding matrices $R, \overline{R}$ and $Q$ are:

\[  R_3=\begin{bmatrix}
 q&0&0&0&0&0&0&0&0\\0&q-q^{-1}&0&1&0&0&0&0&0\\0&0&q-q^{-1}&0&0&0&1&0&0\\0&1&0&0&0&0&0&0&0\\0&0&0&0&q&0&0&0&0\\0&0&0&0&0&q-q^{-1}&0&1&0\\0&0&1&0&0&0&0&0&0\\0&0&0&0&0&1&0&0&0\\0&0&0&0&0&0&0&0&q \end{bmatrix}\] 
 
  \[ \overline{R}_3=\begin{bmatrix}
 q^{-1}&0&0&0&0&0&0&0&0\\0&0&0&1&0&0&0&0&0\\0&0&0&0&0&0&1&0&0\\0&1&0&q^{-1}-q&0&0&0&0&0\\0&0&0&0&q^{-1}&0&0&0&0\\0&0&0&0&0&0&0&1&0\\0&0&1&0&0&0&q^{-1}-q&0&0\\0&0&0&0&0&1&0&q^{-1}-q&0\\0&0&0&0&0&0&0&0&q^{-1} \end{bmatrix}\] 

 \[ Q_3=\begin{bmatrix}
 q+q^{-1}&0&0&0&0&0&0&0&0\\0&q&0&1&0&0&0&0&0\\0&0&q&0&0&0&1&0&0\\0&1&0&q^{-1}&0&0&0&0&0\\0&0&0&0&q+q^{-1}&0&0&0&0\\0&0&0&0&0&q&0&1&0\\0&0&1&0&0&0&q^{-1}&0&0\\0&0&0&0&0&1&0&q^{-1}&0\\0&0&0&0&0&0&0&0&q+q^{-1} \end{bmatrix}\] 
It can be shown that, for any fixed $n \in \mathbb{Z}$, $n\geq 2$, 
\[R_nQ_n= Q_nR_n = q \cdot Q_n \hspace{0.5cm} \text{and} \hspace{0.5cm} \overline{R}_nQ_n= Q_n\overline{R}_n = q^{-1} \cdot Q_n,\] 
which mimic the properties of the polynomial $\brak{G}$ discussed in Section~\ref{sec:inv-singlinks}.

The \textit{singular braid monoid} on $k$ strands, denoted by $SB_k$, is a monoid with generators $\sigma_i, \sigma_i^{-1}$ and $\tau_i$, where $1\leq i\leq k-1$:

\vspace{0.8cm}

$$\begin{picture}(50,15)
\raisebox{0.7cm}{$\sigma_i=$}\,\,
 \includegraphics[height=1.7cm]{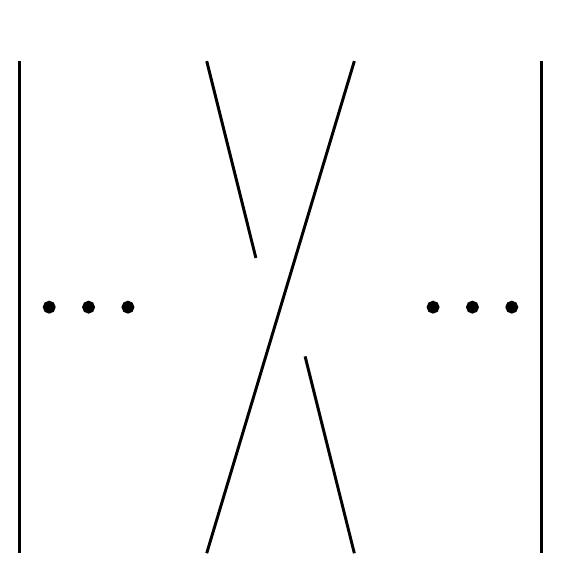} 
 \put(-38,45){\fontsize{10}{11}$i$}
 \put(-28,45){\fontsize{10}{11}$i+1$}\end{picture}
  \hspace{1.7cm}
\raisebox{0.7cm}{$\sigma_i^{-1}=$}\,\, 
\begin{picture}(50,15)\reflectbox{\includegraphics[height=1.7cm]{GenBraid.pdf}}
\put(-38,45){\fontsize{10}{11}$i$}
\put(-28,45){\fontsize{10}{11}$i+1$}\end{picture} 
\hspace{1cm}
\raisebox{0.7cm}{$\tau_i=$}\,\, 
\begin{picture}(50,15)\includegraphics[height=1.7cm]{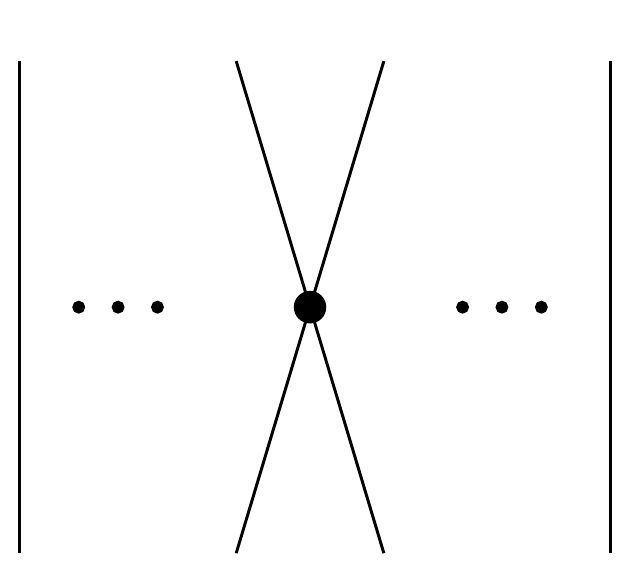}
\put(-38,45){\fontsize{10}{11}$i$}
\put(-28,45){\fontsize{10}{11}$i+1$}\end{picture}$$
and relations:
\begin{itemize}
\item $g_ih_j=h_jg_i$ where $|i-j|>1$ and $g_i,h_i\in \{\sigma_i,\sigma_i^{-1},\tau_i\}$.

\item $\sigma_i^{-1}\sigma_i=1_n=\sigma_i\sigma_i^{-1}\,\, (R2)$  
 \item  $\sigma_i\sigma_j\sigma_i=\sigma_j\sigma_i\sigma_j$, for $|i-j|=1\,\, (R3)$
\item $\tau_i\sigma_j\sigma_i=\sigma_j\sigma_i\tau_j$, for $|i-j|=1\,\, (R4)$ 
\item  $\sigma_i\tau_i=\tau_i\sigma_i \,\,(R5)$
\end{itemize}
Below we depict the last two relations corresponding to the extended Reidemeister moves of type 4 and type 5:
\[ \raisebox{13pt}{\includegraphics[height=1.9cm]{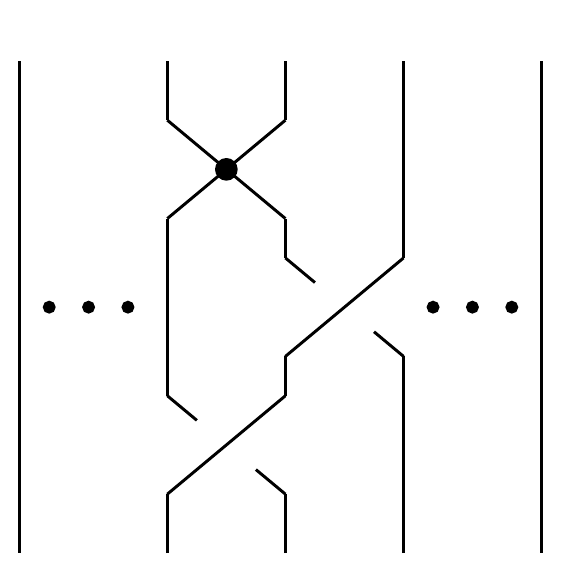}}\hspace{0.2cm} \stackrel{R4}{\raisebox{35pt}{=}}\hspace{0.2cm} \raisebox{13pt}{\includegraphics[height=1.9cm]{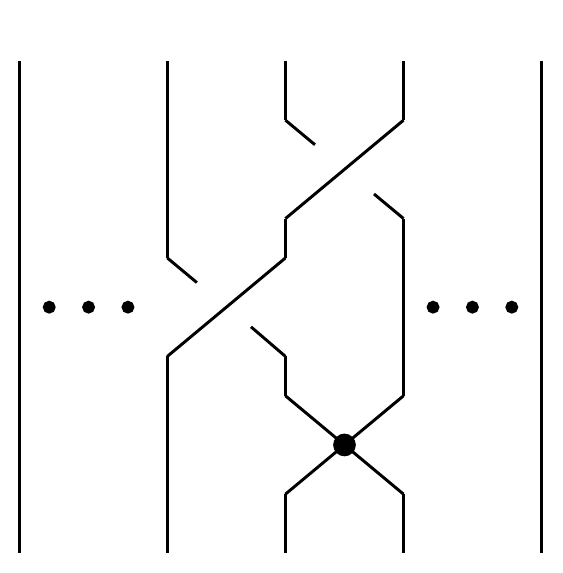}}\hspace{2cm}
\raisebox{13pt}{\includegraphics[height=1.9cm]{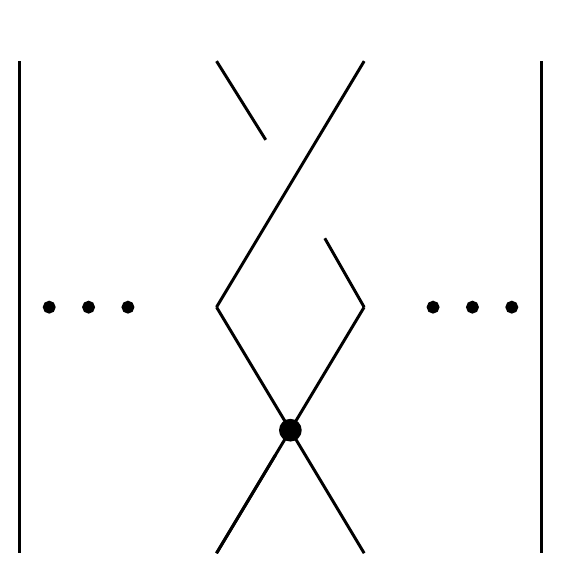}}\hspace{0.2cm} \stackrel{R5}{\raisebox{35pt}{=}}\hspace{0.2cm} \raisebox{13pt}{\includegraphics[height=1.9cm]{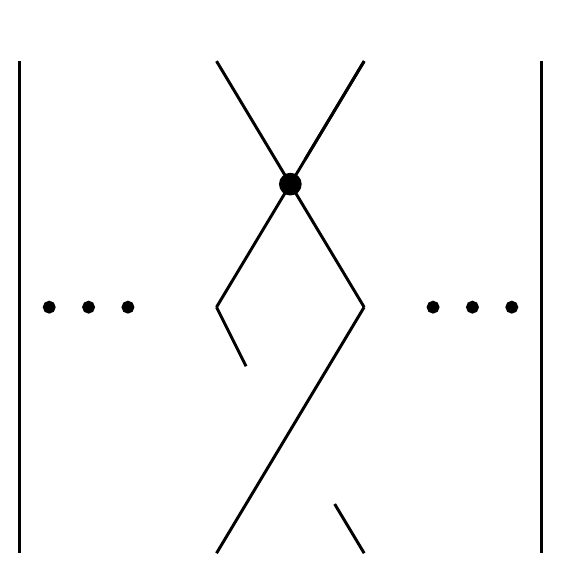}} \]
We orient the singular braids so that all strands are oriented downward. Next, we employ the matrices $R_n, \overline{R}_n$ and $Q_n$ to define, for every $n\in \mathbb{N}$, $n\geq 2$, a homomorphism $\rho_n$ from $SB_k$ into a matrix algebra over $\Z[q, q^{-1}]$, given by:

\[\sigma_1 \mapsto R_n \otimes I \otimes \dots \otimes I\hspace{2cm} \sigma_1^{-1} \mapsto \overline{R}_n \otimes I \otimes \dots \otimes I\  \]
\[\sigma_2 \mapsto I \otimes R_n \otimes \dots \otimes I  \hspace{2cm} \sigma_2^{-1} \mapsto I \otimes \overline{R}_n \otimes \dots \otimes I    \]
\[ \dots \]
\[ \sigma_{k-1} \mapsto I \otimes I \otimes \dots \otimes R_n \hspace{1.8cm} \sigma_{k-1}^{-1} \mapsto I \otimes I \otimes I \dots \otimes \overline{R}_n      \]

\[\tau_1 \mapsto Q_n \otimes I \otimes \dots \otimes I \]
\[\tau_2 \mapsto I \otimes Q_n \otimes \dots \otimes I \]
\[ \dots \]
\[ \tau_{k-1} \mapsto I \otimes I \otimes \dots \otimes Q_n,\]
Here, $\otimes$ is the Kronecker delta tensor product of matrices. Recall that if $A$ is an $m\times n$ matrix and $B$ is a $p \times q$ matrix, then their Kronecker product $A \otimes B$ is the $mp \times nq$ block matrix given below: 
\[ A \otimes B = \left [\begin{array} {ccc}
a_{11}B & \dots & a_{1n}B\\
\vdots & \ddots & \vdots \\
a_{m1}B & \dots & a_{mn}B
 \end{array} \right ].  \]

Notice that for a $k$-stranded singular braid $\beta$, the associated square matrix $\rho_n(\beta)$ with entries in $\Z[q, q^{-1}]$ has size $n^k \times n^k$.
Since the polynomial $\brak{G}$ is a regular isotopy invariant for singular links, it implies that the mapping $\rho_n$ preserves the last four singular braid monoid relations. A close look also reveals that for $|i-j|>1$, $\rho_n(g_ih_j) = \rho_n(h_jg_i)$, where $g_i,h_i\in \{\sigma_i,\sigma_i^{-1},\tau_i\}$. This equality holds since the resulting matrices (on both sides of the equality), written as a Kronecker delta tensor product of matrices, will contain the same matrices ($R_n, \overline{R}_n$ or $Q_n$) on the $i$th and $j$th components, respectively, and the $n\times n$ identity matrix on the other components of the tensor product. Therefore, the following statement holds.

\begin{theorem}
For every $n\in \mathbb{Z}$, $n\geq 2$, the mapping $\rho_n$ is a representation of the singular braid monoid $SB_k$ into a matrix algebra over $\Z[q, q^{-1}]$. 
\end{theorem}

\section{Yet another look at $sl(n)$ invariants}\label{sec:MOYrelations}

In this section we show that the polynomial invariant for singular links constructed in Section~\ref{sec:inv-singlinks} can be used to obtain a version of the Murakami-Ohtsuki-Yamada (MOY) state model for the $sl(n)$ polynomial (for details on this state model we refer the reader to~\cite{MOY}). In other words, by extending the Yang-Baxter state model for the $sl(n)$-link invariant to singular links we obtain a state model for the $sl(n)$ polynomial, defined via a graphical calculus of planar 4-valent graphs.

We start off with a handy statement, which will be used to derive a set of skein relations involving only planar graphs. 

\begin{proposition}\label{prop:useful}
The following skein relations hold:             
\[ \left<  \raisebox{-12pt}{ \includegraphics[height=.4in]{flatcross}}\,\right>  =   \left< \raisebox{-12pt}{ \includegraphics[height=.4in]{pcross}}\,\right>  +q^{-1}  \left< \raisebox{-12pt}{  \includegraphics[height=.4in]{2arcs} }\, \right > =   \left< \raisebox{-12pt}{ \includegraphics[height=.4in]{ncross}}\,\right>  +q  \left< \raisebox{-12pt}{  \includegraphics[height=.4in]{2arcs} }\, \right > .\]
\end{proposition}

\begin{proof} The statement follows from the skein relations in Figure~\ref{fig:crossings} and Figure~\ref{fig:sing-crossing}, as we show below.

\begin{eqnarray*}
 \left< \raisebox{-12pt}{ \includegraphics[height=.4in]{pcross}}\,\right>& +& q^{-1}  \left< \raisebox{-12pt}{  \includegraphics[height=.4in]{2arcs} }\, \right > \\
 &=&(q-q^{-1})\left< \raisebox{-13pt}{ \includegraphics[height=.4in]{ltsplit}}\, \right>
      +q \left<  \raisebox{-13pt}{ \includegraphics[height=.4in]{eqsplit}}\, \right>
      + \left<  \raisebox{-13pt}{  \includegraphics[height=.4in]{flat}}
              \put(-18, 7){\fontsize{10}{11}$ \neq$}\, \right> \\ \\ 
 && + q^{-1}\left < \raisebox{-13pt}{ \includegraphics[height=.4in]{ltsplit}}\, \right> + q^{-1}\left < \raisebox{-13pt}{ \includegraphics[height=.4in]{gtsplit}}\, \right> + q^{-1}\left < \raisebox{-13pt}{ \includegraphics[height=.4in]{eqsplit}}\, \right> \\
 &=& q \left< \raisebox{-13pt}{ \includegraphics[height=.4in]{ltsplit}}\, \right> + q^{-1} \left< \raisebox{-13pt}{ \includegraphics[height=.4in]{gtsplit}}\, \right> +(q+q^{-1}) \left< \raisebox{-13pt}{ \includegraphics[height=.4in]{eqsplit}}\, \right> + \left<  \raisebox{-13pt}{  \includegraphics[height=.4in]{flat}}  \put(-18, 7){\fontsize{10}{11}$ \neq$}\, \right> \\
 &=& \left< \raisebox{-13pt}{ \includegraphics[height=.4in]{flatcross}}\, \right>.
\end{eqnarray*}
The second equality can be verified similarly, or by using the first equality together with the exchange skein relation defining the $sl(n)$-link invariant, as we explain below.
\begin{eqnarray*}
\left< \raisebox{-13pt}{ \includegraphics[height=.4in]{flatcross}}\, \right> &=&  \left< \raisebox{-12pt}{ \includegraphics[height=.4in]{pcross}}\,\right>  + q^{-1}  \left< \raisebox{-12pt}{  \includegraphics[height=.4in]{2arcs} }\, \right >\\
&=& \left[ \left< \raisebox{-12pt}{ \includegraphics[height=.4in]{ncross}}\,\right> + (q-q^{-1}) \left< \raisebox{-12pt}{ \includegraphics[height=.4in]{2arcs}}\,\right> \right] + q^{-1}  \left< \raisebox{-12pt}{  \includegraphics[height=.4in]{2arcs} }\, \right >\\
&=&  \left< \raisebox{-12pt}{ \includegraphics[height=.4in]{ncross}}\,\right> +q \left< \raisebox{-12pt}{  \includegraphics[height=.4in]{2arcs} }\, \right >.
\end{eqnarray*}
\end{proof}

\begin{proposition}\label{prop:graph skein rel}
The following graph skein relations hold:

\begin{eqnarray}\label{eq:skein-reid1}
 \left< \raisebox{-13pt}{ \includegraphics[height=.4in]{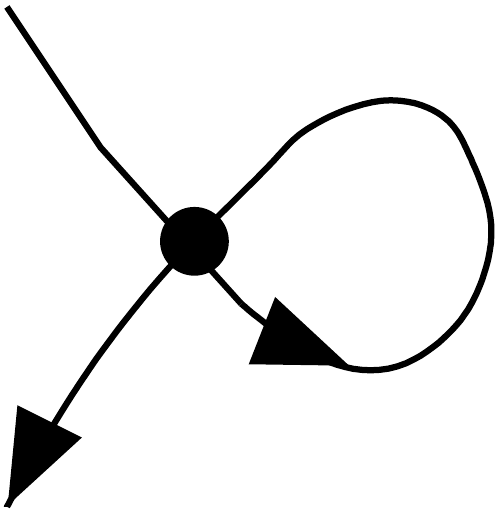}}\,\right>  =  [n+1]   \left< \raisebox{-13pt}{ \includegraphics[height=.4in]{arc}}\,\right>
 \end{eqnarray}
 
 \begin{eqnarray} \label{eq:skein-reid2a}
  \left< \raisebox{-13pt}{ \includegraphics[height=.4in]{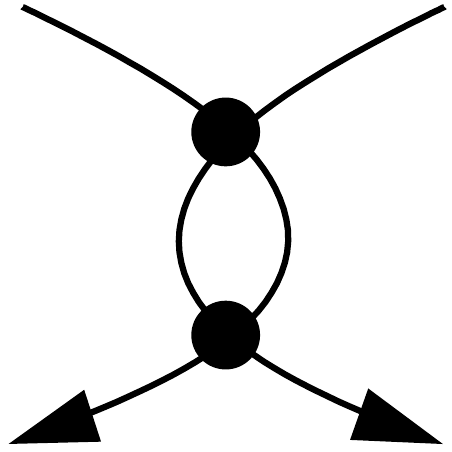}}\,\right>  = [2]  \left< \raisebox{-13pt}{  \includegraphics[height=.4in]{flatcross} } \right> 
 \end{eqnarray}

\begin{eqnarray}\label{eq:skein-reid2b}
 \left< \raisebox{-13pt}{ \includegraphics[height=.4in]{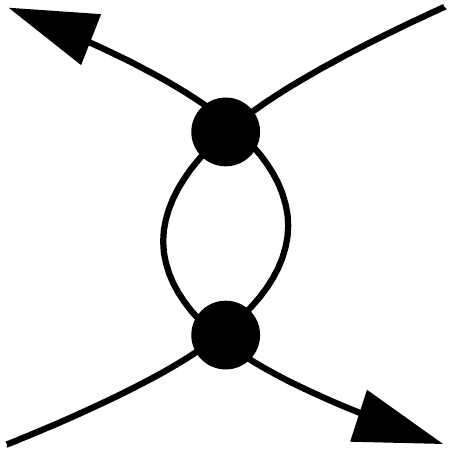}}\,\right>  =  \left< \raisebox{-13pt}{ \includegraphics[height=.4in]{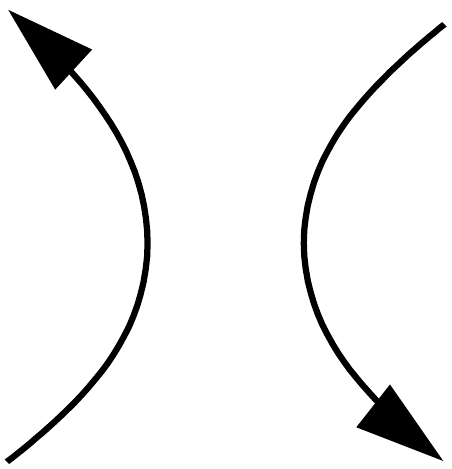}}\,\right> + [n +2]  \left<\, \reflectbox{\raisebox{-10pt}{ \includegraphics[height=.4in, angle=90]{2arcs-op}}}\,\right> 
 \end{eqnarray}

\begin{eqnarray}\label{eq:skein-reid3a}
 \left< \raisebox{-13pt}{ \includegraphics[height=.4in]{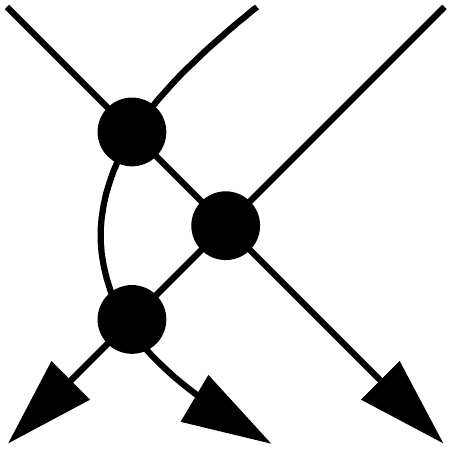}}\,\right>  + \left< \raisebox{-13pt}{ \includegraphics[height=.4in]{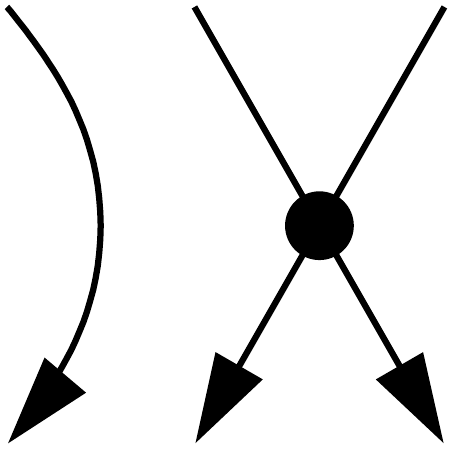}}\,\right>  =  \left<\, \reflectbox{ \raisebox{-13pt}{ \includegraphics[height=.4in]{flatreid3a}}} \hspace{-0.05in}\right>  + \left<\, \reflectbox{ \raisebox{-13pt}{ \includegraphics[height=.4in]{vertex-arc}}}  \hspace{-0.05in}\right>   
 \end{eqnarray}

\begin{eqnarray} \label{eq:skein-reid3b}
  \left< \raisebox{-13pt}{ \includegraphics[height=.4in]{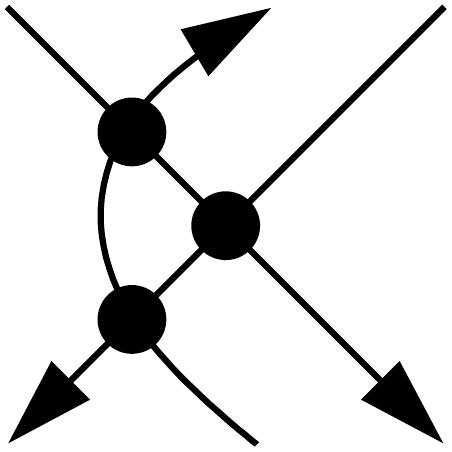}}\,\right>  -[n+3] \left< \raisebox{-13pt}{ \includegraphics[height=.4in]{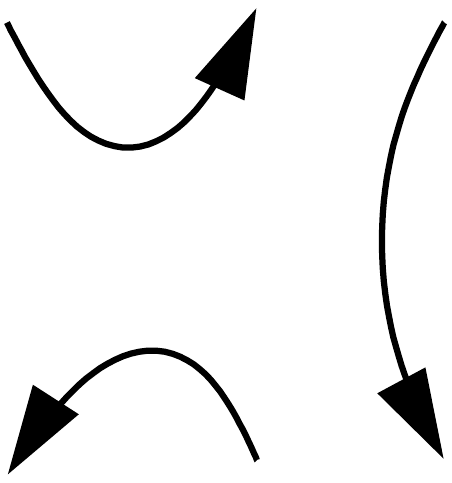}}\,\right>  =  \left<\, \reflectbox{ \raisebox{-13pt}{ \includegraphics[height=.4in]{flatreid3b}}} \hspace{-0.05in}\right>  -[n+3] \left<\, \reflectbox{ \raisebox{-13pt}{ \includegraphics[height=.4in]{skein-last}}}  \hspace{-0.05in}\right>.  
\end{eqnarray}

\end{proposition}

\begin{proof}
We will make use of the skein relations in Proposition~\ref{prop:useful}. We start with the first skein relation:
\begin{eqnarray*}
 \left< \raisebox{-13pt}{ \includegraphics[height=.4in]{flatkink}}\,\right>  &=&  \left< \raisebox{-13pt}{ \includegraphics[height=.4in]{poskink}}\,\right> + q^{-1}\left< \raisebox{-13pt}{ \includegraphics[height=.4in]{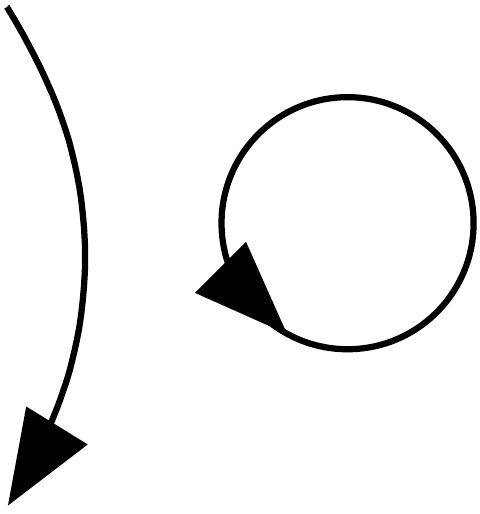}}\,\right>\\
 &=& q^n  \left< \raisebox{-13pt}{ \includegraphics[height=.4in]{arc}}\,\right>  + q^{-1} [n]  \left< \raisebox{-13pt}{ \includegraphics[height=.4in]{arc}}\,\right> \\
 &=& [n+1]   \left< \raisebox{-13pt}{ \includegraphics[height=.4in]{arc}}\,\right>.
 \end{eqnarray*}
 The skein relation in Equation~\eqref{eq:skein-reid2a} is verified as follows:
 
 \begin{eqnarray*}
  \left< \raisebox{-13pt}{ \includegraphics[height=.4in]{flatreid2a}}\,\right>  &=& \left< \raisebox{-13pt}{ \includegraphics[height=.4in]{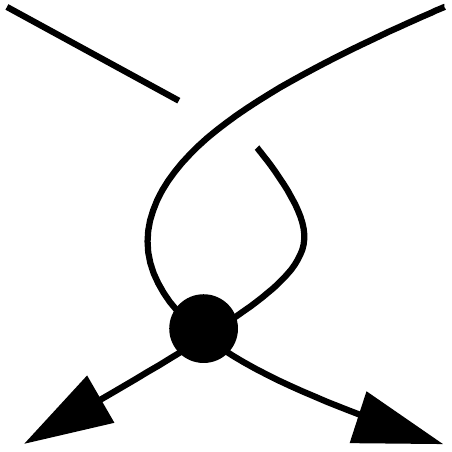}}\,\right>  +q^{-1}  \left< \raisebox{-13pt}{ \includegraphics[height=.4in]{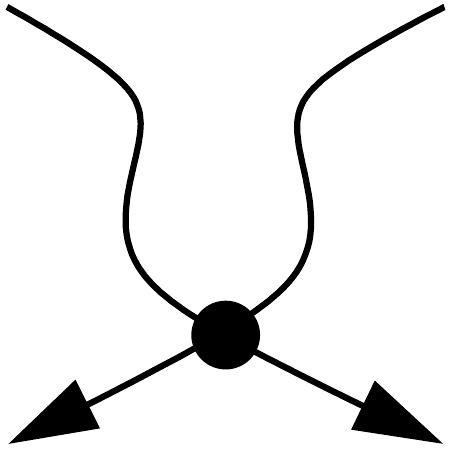}}\,\right>\\
&=& \left< \raisebox{-13pt}{ \includegraphics[height=.4in]{reid2a}}\,\right> +q  \left< \raisebox{-13pt}{ \includegraphics[height=.4in]{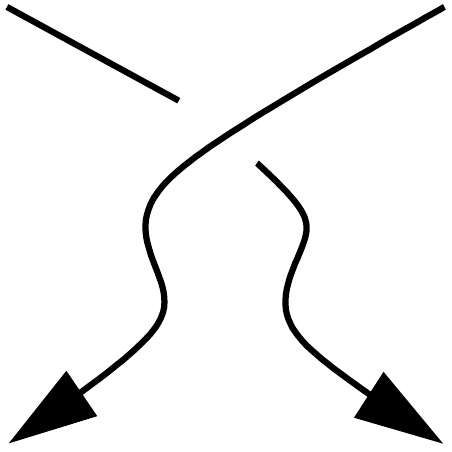}}\,\right>  +q^{-1}  \left< \raisebox{-13pt}{ \includegraphics[height=.4in]{pfskein-2}}\,\right>\\  
&=& \left< \raisebox{-13pt}{ \includegraphics[height=.4in]{2arcs}}\,\right> +q  \left< \raisebox{-13pt}{ \includegraphics[height=.4in]{pcross}}\,\right>  +q^{-1}  \left< \raisebox{-13pt}{ \includegraphics[height=.4in]{flatcross}}\,\right>\\  
&=&q\left( q^{-1} \left< \raisebox{-13pt}{ \includegraphics[height=.4in]{2arcs}}\,\right> +  \left< \raisebox{-13pt}{ \includegraphics[height=.4in]{pcross}}\,\right> \right) +q^{-1}  \left< \raisebox{-13pt}{ \includegraphics[height=.4in]{flatcross}}\,\right>\\
&=&q \left< \raisebox{-13pt}{ \includegraphics[height=.4in]{flatcross}}\,\right>+q^{-1}  \left< \raisebox{-13pt}{ \includegraphics[height=.4in]{flatcross}}\,\right>\\  
&=&[2] \left< \raisebox{-13pt}{ \includegraphics[height=.4in]{flatcross}}\,\right>.
  \end{eqnarray*}
  
  We consider now the third skein relation in Equation~\eqref{eq:skein-reid2b}:
   \begin{eqnarray*}
   \left< \raisebox{-13pt}{ \includegraphics[height=.4in]{flatreid2b}}\,\right> &=& \left< \reflectbox{\raisebox{18pt}{ \includegraphics[height=.45in, angle=180]{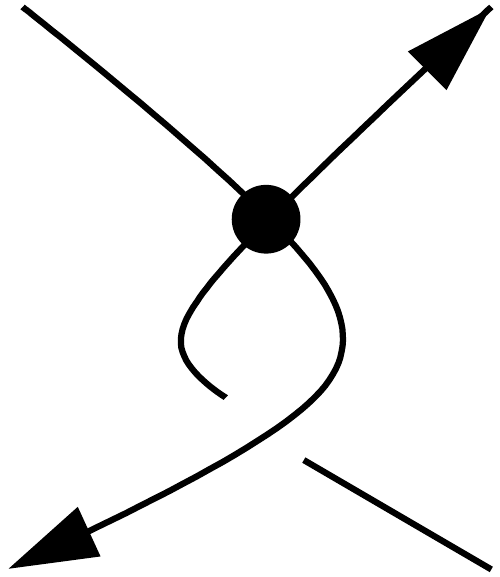}}}\,\right>  + q^{-1} \left< \raisebox{-13pt}{ \includegraphics[height=.35in, angle=90]{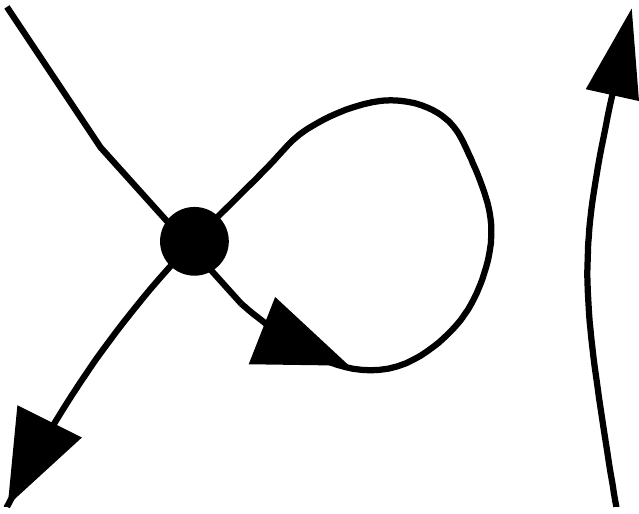}}\,\right>\\  
   &=& \left< \reflectbox{\raisebox{-13pt}{ \includegraphics[height=.4in]{reid2b}}}\,\right>  + q  \left< \raisebox{18pt}{ \includegraphics[height=.35in, angle=270]{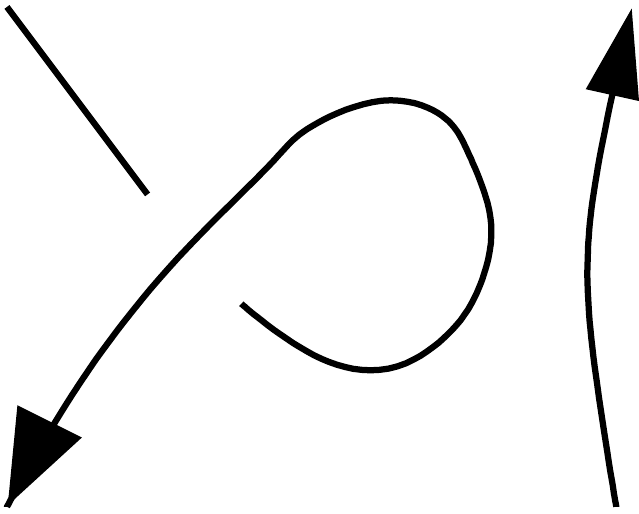}}\,\right> + q^{-1} \left< \raisebox{-13pt}{ \includegraphics[height=.35in, angle=90]{pfskein-4}}\,\right>\\  
   &=&  \left< \raisebox{-13pt}{ \includegraphics[height=.4in]{2arcs-op}}\,\right> + q\cdot q^n  \left<\, \reflectbox{\raisebox{-10pt}{ \includegraphics[height=.4in, angle=90]{2arcs-op}}}\right>  + q^{-1}\cdot [n+1]  \left<\, \reflectbox{\raisebox{-10pt}{ \includegraphics[height=.4in, angle=90]{2arcs-op}}}\right>  \\
   &=& \left< \raisebox{-13pt}{ \includegraphics[height=.4in]{2arcs-op}}\,\right> + [n +2]  \left<\, \reflectbox{\raisebox{-10pt}{ \includegraphics[height=.4in, angle=90]{2arcs-op}}}\right>.  
     \end{eqnarray*}

  In the latter computations we used Equation~\eqref{eq:skein-reid1}, the invariance of the polynomial under the second Reidemeister move, and the behavior of the polynomial under the first Reidemeister move. 
  
We are left with showing the skein relations in Equations~\eqref{eq:skein-reid3a} and~\eqref{eq:skein-reid3b}. For that, we first show that the following identities hold:
\begin{eqnarray}\label{eq:reid3a-mixed}
\hspace{1.5 cm}  \left< \raisebox{-13pt}{ \includegraphics[height=.4in]{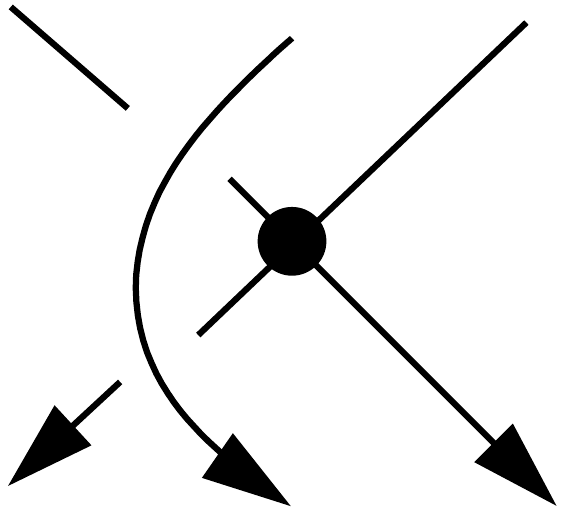}}\right> = \left<\reflectbox{ \raisebox{-13pt}{ \includegraphics[height=.4in]{reid3a-mixed}}} \right> \hspace{0.5cm} \text{and} \hspace{0.5cm}  \left< \raisebox{-13pt}{ \includegraphics[height=.4in]{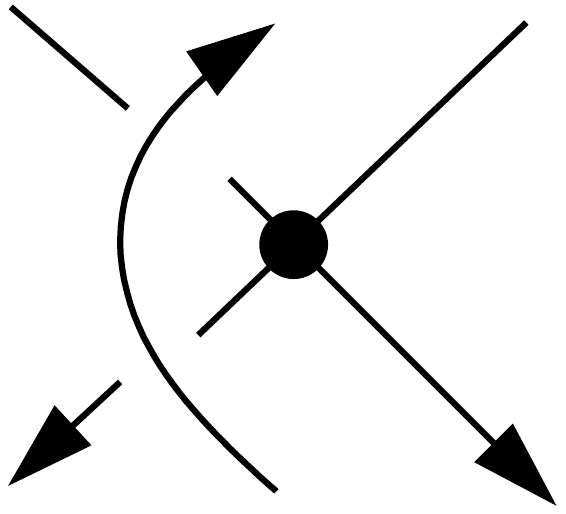}}\right> = \left<\reflectbox{ \raisebox{-13pt}{ \includegraphics[height=.4in]{reid3b-mixed}}} \right>.
 \end{eqnarray}

The first identity in~\eqref{eq:reid3a-mixed} is verified as shown below; the second identity is verified in a similar manner, and thus it is omitted.
  \begin{eqnarray*}
\left< \raisebox{-13pt}{ \includegraphics[height=.4in]{reid3a-mixed}}\right> &=&\left< \raisebox{-13pt}{ \includegraphics[height=.4in]{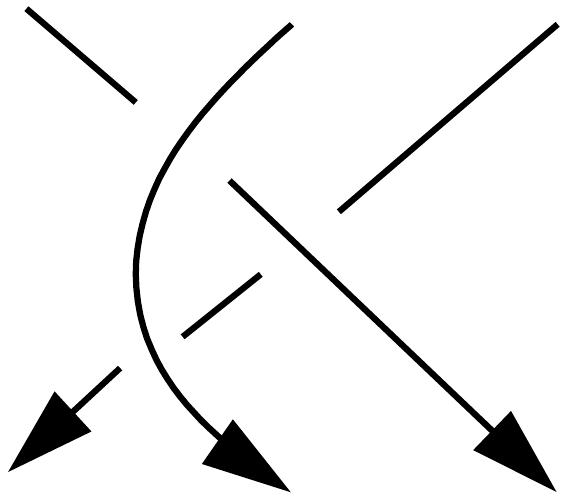}}\right> + q  \left< \raisebox{-13pt}{ \includegraphics[height=.4in]{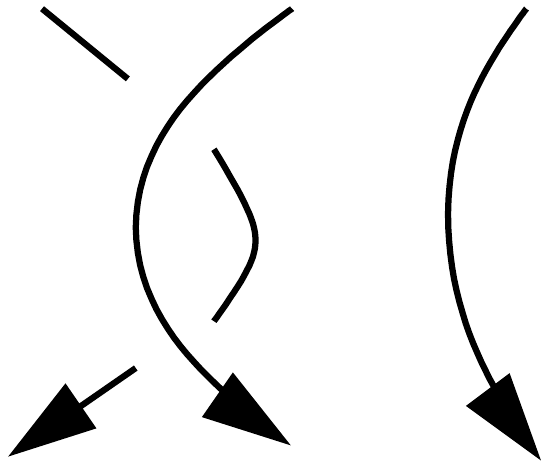}}\right> \\
& \stackrel{R3, R2}{=} & \left< \reflectbox{\raisebox{-13pt}{ \includegraphics[height=.4in]{pfskein-7}}} \right> + q  \left< \reflectbox{\raisebox{-13pt}{ \includegraphics[height=.4in]{pfskein-6}}} \right>\\
&=& \left<\reflectbox{ \raisebox{-13pt}{ \includegraphics[height=.4in]{reid3a-mixed}}} \right>.
 \end{eqnarray*}
Then we have:
 \begin{eqnarray*}
  \left< \raisebox{-13pt}{ \includegraphics[height=.4in]{flatreid3a}}\,\right> &=& \left< \raisebox{-13pt}{ \includegraphics[height=.4in]{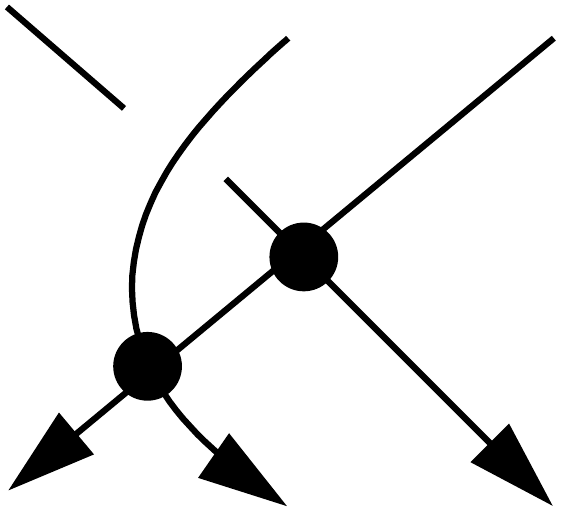}}\right> + q^{-1}  \left< \raisebox{-13pt}{ \includegraphics[height=.4in]{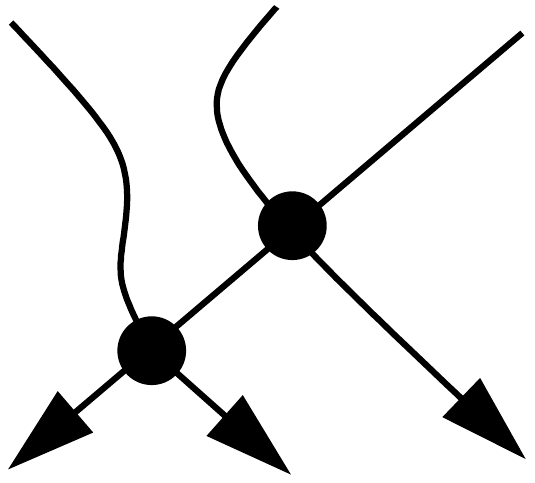}}\right> \\
  &=& \left< \raisebox{-13pt}{ \includegraphics[height=.4in]{reid3a-mixed}}\right> + q  \left< \raisebox{-13pt}{ \includegraphics[height=.4in]{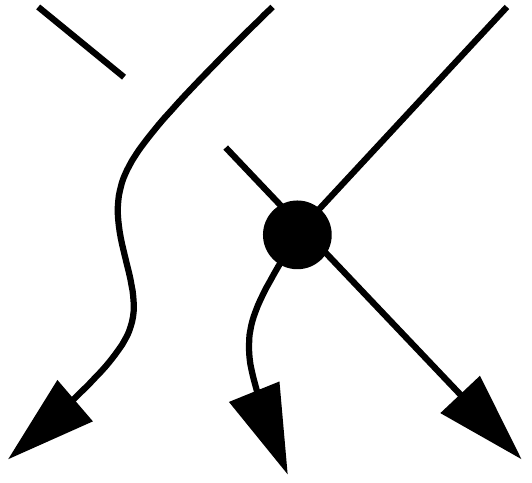}}\right>  +q^{-1}  \left< \raisebox{-13pt}{ \includegraphics[height=.4in]{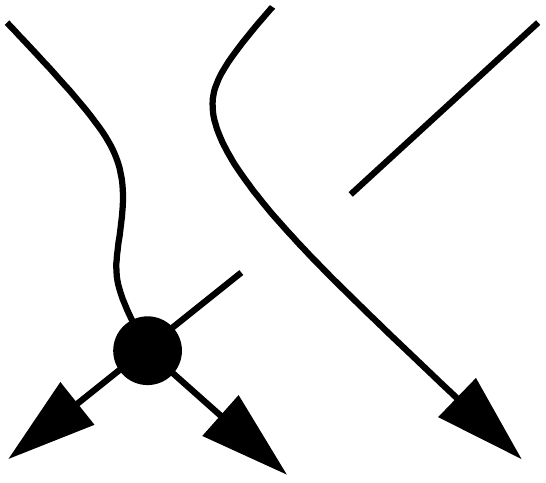}}\right> + \left< \raisebox{-13pt}{ \includegraphics[height=.4in]{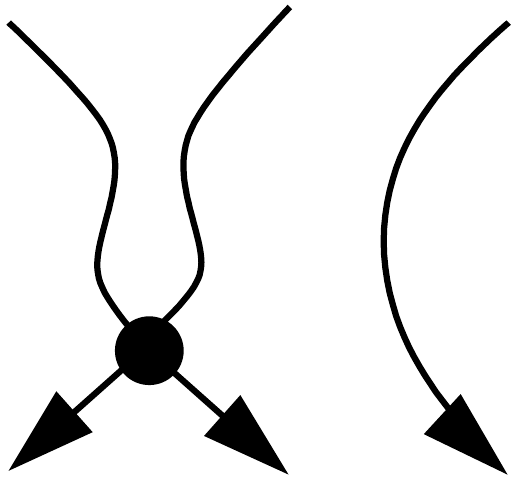}}\right> ,
 \end{eqnarray*}
 and 
 \begin{eqnarray*}
  \left< \reflectbox{\raisebox{-13pt}{ \includegraphics[height=.4in]{flatreid3a}}}\,\right> &=& \left< \reflectbox{\raisebox{-13pt}{ \includegraphics[height=.4in]{pfskein-8}}}\right> + q \left< \reflectbox{\raisebox{-13pt}{ \includegraphics[height=.4in]{pfskein-9}}}\right> \\
  &=& \left< \reflectbox{ \raisebox{-13pt}{ \includegraphics[height=.4in]{reid3a-mixed}}} \right> + q^{-1}  \left< \reflectbox{\raisebox{-13pt}{ \includegraphics[height=.4in]{pfskein-10}}}\right>  +q  \left< \reflectbox{\raisebox{-13pt}{ \includegraphics[height=.4in]{pfskein-11}}}\right> + \left< \reflectbox{ \raisebox{-13pt}{ \includegraphics[height=.4in]{pfskein-12}}}\right>.
 \end{eqnarray*}
After applying planar isotopies to some of the diagrams above, the skein relation given in Equation~\eqref{eq:skein-reid3a} follows. 

Finally, we verify the skein relation depicted in Equation~\eqref{eq:skein-reid3b}.
 \begin{eqnarray*}
  \left< \raisebox{-13pt}{ \includegraphics[height=.4in]{flatreid3b}}\,\right> &=& \left< \raisebox{-13pt}{ \includegraphics[height=.4in]{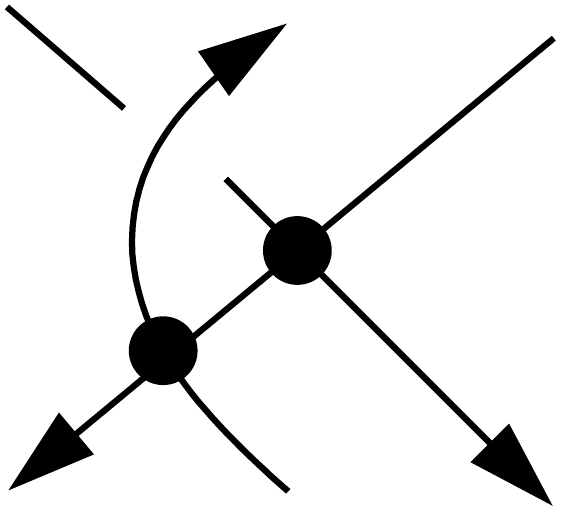}}\right> + q  \left< \raisebox{-13pt}{ \includegraphics[height=.4in]{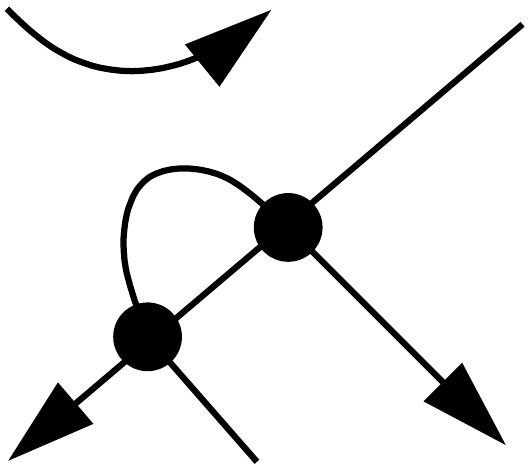}}\right> \\
 &=&  \left< \raisebox{-13pt}{ \includegraphics[height=.4in]{reid3b-mixed}}\right> + q^{-1} \left< \raisebox{-13pt}{ \includegraphics[height=.4in]{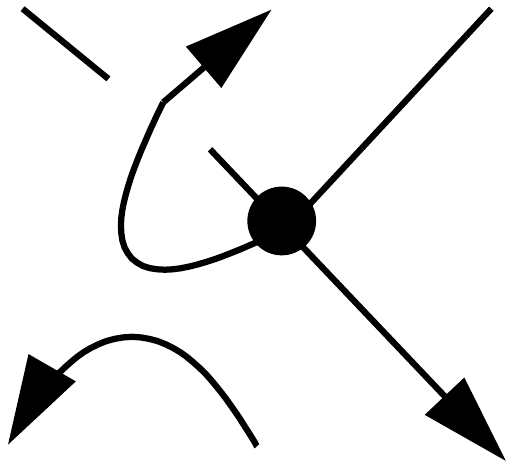}}\right> + q  \left< \raisebox{-13pt}{ \includegraphics[height=.4in]{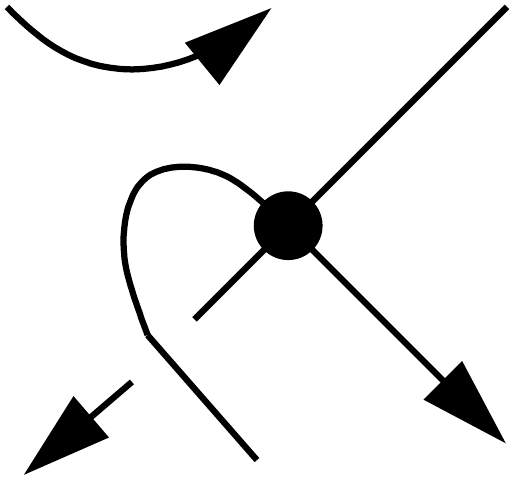}}\right> +   \left< \raisebox{-13pt}{ \includegraphics[height=.4in]{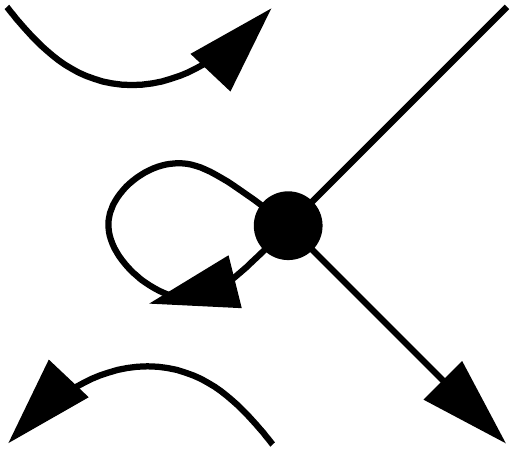}}\right>\\
  &=&  \left< \raisebox{-13pt}{ \includegraphics[height=.4in]{reid3b-mixed}}\right> + q^{-1} \left( \left< \raisebox{-13pt}{ \includegraphics[height=.4in]{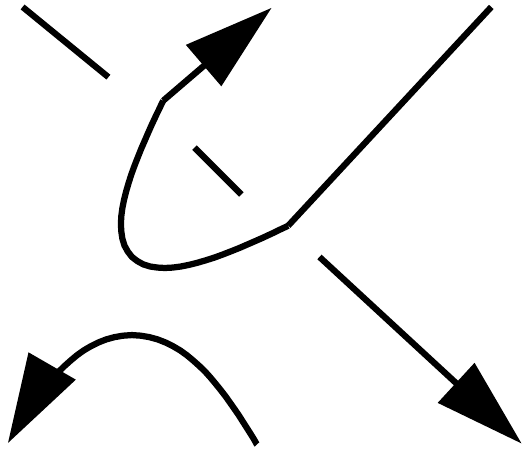}}\right> + q^{-1}  \left< \raisebox{-13pt}{ \includegraphics[height=.4in]{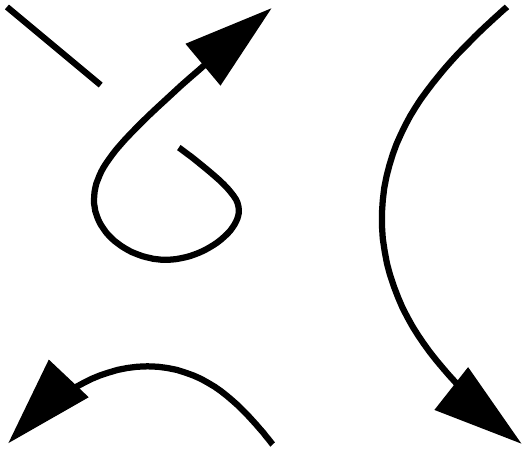}}\right> \right ) \\
  && + q  \left( \left< \raisebox{-13pt}{ \includegraphics[height=.4in]{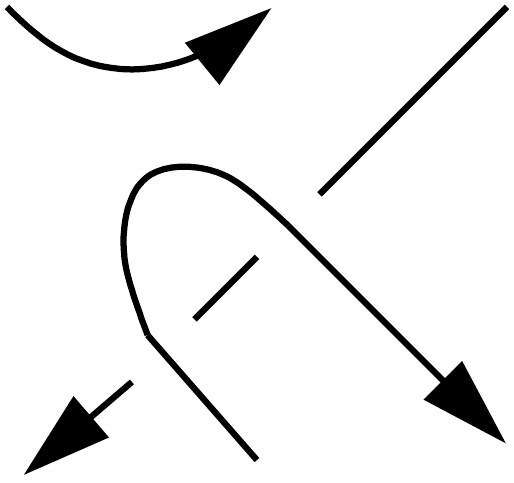}}\right>  + q  \left< \raisebox{-13pt}{ \includegraphics[height=.4in]{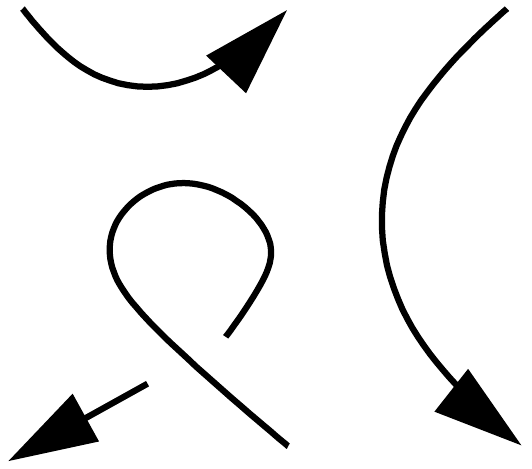}}\right>    \right) + [n+1] \left< \raisebox{-13pt}{ \includegraphics[height=.4in]{skein-last}}\,\right>  \\
   &=&  \left< \raisebox{-13pt}{ \includegraphics[height=.4in]{reid3b-mixed}}\right> + q^{-1} \left< \raisebox{-13pt}{ \includegraphics[height=.4in]{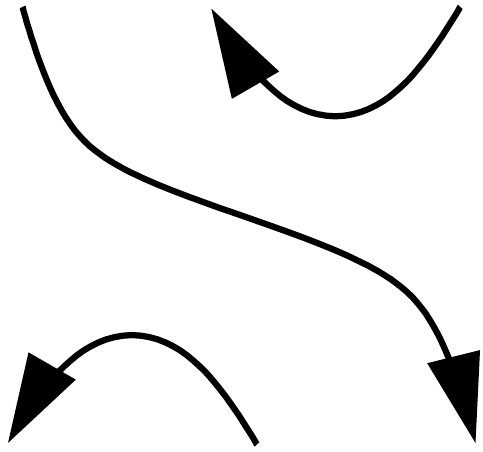}}\right> + q^{-2} q^{-n}\left< \raisebox{-13pt}{ \includegraphics[height=.4in]{skein-last}}\,\right> \\
   && + q \left< \reflectbox{\raisebox{-13pt}{ \includegraphics[height=.4in]{pfskein-22}}}\right> + q^2 q^n \left< \raisebox{-13pt}{ \includegraphics[height=.4in]{skein-last}}\,\right> + [n+1] \left< \raisebox{-13pt}{ \includegraphics[height=.4in]{skein-last}}\,\right>\\
   &=&  \left< \raisebox{-13pt}{ \includegraphics[height=.4in]{reid3b-mixed}}\right>  + q^{-1} \left< \raisebox{-13pt}{ \includegraphics[height=.4in]{pfskein-22}}\right> + q \left< \reflectbox{\raisebox{-13pt}{ \includegraphics[height=.4in]{pfskein-22}}}\right> + [n+3] \left< \raisebox{-13pt}{ \includegraphics[height=.4in]{skein-last}}\,\right>,
  \end{eqnarray*}
where we used that $q^{-n-2}+ q^{n+2} + [n+1] = [n+3]$. Similar computations reveal that
 \begin{eqnarray*}
   \left< \reflectbox{\raisebox{-13pt}{ \includegraphics[height=.4in]{flatreid3b}}}\,\right> &=& \left< \reflectbox{\raisebox{-13pt}{ \includegraphics[height=.4in]{reid3b-mixed}}}\right>  + q^{-1} \left< \raisebox{-13pt}{ \includegraphics[height=.4in]{pfskein-22}}\right> + q \left< \reflectbox{\raisebox{-13pt}{ \includegraphics[height=.4in]{pfskein-22}}}\right> + [n+3] \left< \reflectbox{\raisebox{-13pt}{ \includegraphics[height=.4in]{skein-last}}}\,\right>.
   \end{eqnarray*}
Employing the second identity in~\eqref{eq:reid3a-mixed}, we see that the desired skein relation in Equation~\eqref{eq:skein-reid3b} holds.
\end{proof}

\begin{remark}
The graph skein relations given in Proposition~\ref{prop:graph skein rel} are consistent and sufficient to assign in a unique way a Laurent polynomial in $\mathbb{Z}[q, q^{-1}]$ to any 4-valent planar graph with crossing-type oriented vertices. (Compare with the work in~\cite{KV}.)
\end{remark}

Given a link diagram $D$ (or singular link diagram $G$) we can write each classical crossing in $D$ (or in $G$) as follows:
\begin{eqnarray*}
\left<  \raisebox{-12pt}{ \includegraphics[height=.4in]{pcross}}\,\right>  &=& \left< \raisebox{-12pt}{ \includegraphics[height=.4in]{flatcross}}\,\right>  - q^{-1}  \left< \raisebox{-12pt}{  \includegraphics[height=.4in]{2arcs} }\, \right > \\
\left<  \raisebox{-12pt}{ \includegraphics[height=.4in]{ncross}}\,\right>  &=& \left< \raisebox{-12pt}{ \includegraphics[height=.4in]{flatcross}}\,\right>  - q  \left< \raisebox{-12pt}{  \includegraphics[height=.4in]{2arcs} }\, \right >.
\end{eqnarray*}
This process results in writing $\brak{D}$ (or $\brak{G}$) as a $\mathbb{Z}[q, q^{-1}]$-linear combination of evaluations of planar 4-valent graphs with crossing-type oriented vertices. Then, we evaluate the resulting planar graphs using the graphs skein relations in Proposition~\ref{prop:graph skein rel}, and recover the regular isotopy version of the $sl(n)$ polynomial (or our polynomial invariant for singular links constructed in Section~\ref{sec:inv-singlinks}).

Therefore, this approach provides another method for computing the $sl(n)$ polynomial for oriented knots and links and its extension to singular links.

\begin{remark}
The graphical calculus provided in Proposition~\ref{prop:graph skein rel} is a version of the MOY state model for the $sl(n)$ polynomial given in ~\cite{MOY}, where the wide edges labeled 2 are contracted to result in our 4-valent crossing-type oriented vertices.
\end{remark}

\section{Balanced oriented knotted graphs}\label{sec:knotted graphs}

We would like to see whether we can extend our polynomial invariant for singular links constructed in Section~\ref{sec:inv-singlinks} (and based on a solution of the YBE) to an invariant that includes oriented knotted graphs. Specifically, the following question arises: Can the polynomial $\brak{G} \in \Z[q, q^{-1}]$ be extended so that we obtain an invariant under all versions of the type 6 Reidemeister move shown below? 
\[ \raisebox{-13pt}{ \reflectbox{\includegraphics[height=0.42in]{r61}}}\hspace{0.5cm} \sim \hspace{0.5cm} \raisebox{-13pt}{ \includegraphics[height=.4in]{flatcross}}
 \hspace{0.5cm} \sim \hspace{0.5cm} \raisebox{-13pt}{ \includegraphics[height=0.42in]{r61}}\]
\[{\reflectbox{\raisebox{-13 pt}{\includegraphics[height=.45in,angle =90]{altr5} }}} \hspace{0.5cm} \sim \hspace{0.5cm}  \raisebox{-13 pt}{\includegraphics[height=.4in]{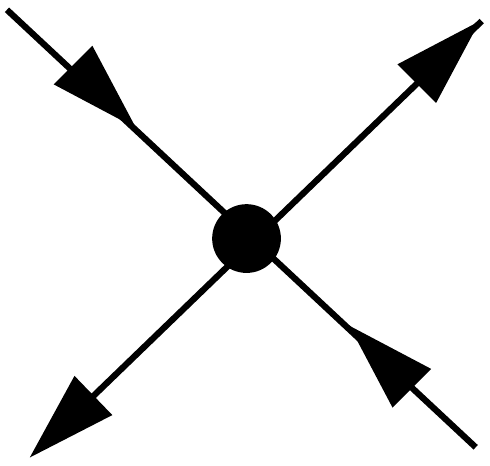} } \hspace{0.5cm} \sim \hspace{0.5cm}
{\reflectbox{\raisebox{-13 pt}{\includegraphics[height=.45in,angle =90]{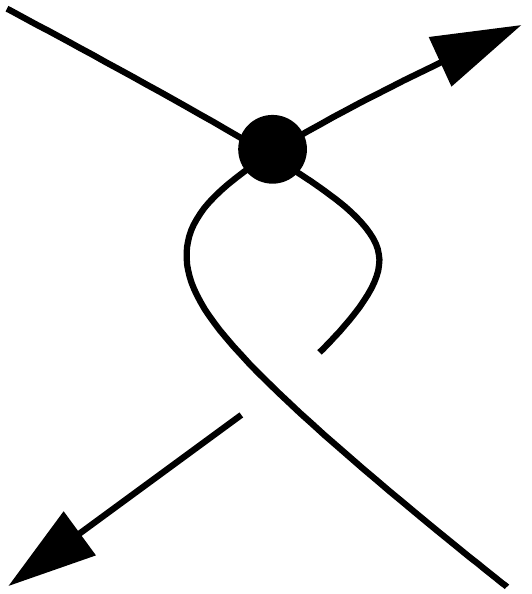}} }}\]
Therefore, we need to consider \textit{balanced oriented} knotted graphs containing not only crossing-type oriented vertices but also alternating oriented vertices:
\[ \raisebox{-13pt}{ \includegraphics[height=.4in]{flatcross}} \hspace{2cm} \raisebox{-14 pt}{\includegraphics[height=.4in]{flatalt} }\]
We will denote the extended polynomial by $[ \,\,\cdot \,\,]$, and we impose the skein relation
\[ \left[\, \raisebox{-13pt}{\includegraphics[height=.4in]{flatalt}}\, \right] =     \gamma \left [\raisebox{-13pt}{ \includegraphics[height=.4in]{2arcs-op2}}\, \right] + \gamma \left [\, \reflectbox{\raisebox{-13pt}{ \includegraphics[height=.4in, angle = 90]{2arcs-op2}}} \right]
 \]
 for some $\gamma \in \mathbb{Z}[q, q^{-1}]$. We also impose that $[ \,\,\cdot \,\,]$ satisfies the skein relations given in Figures~\ref{fig:crossings} and~\ref{fig:sing-crossing}. That is, if $G$ is a singular link diagram, then $[G] := \brak{G}$.
 
 \begin{theorem}
 The polynomial $[\,\, \cdot \,\,]$ is a regular isotopy invariant for balanced oriented knotted graphs with rigid vertices. 
  \end{theorem}
  \begin{proof}
 Because $[\,\, \cdot \,\,]$ satisfies the skein relations given in Figures~\ref{fig:crossings} and~\ref{fig:sing-crossing}, it is invariant under the moves $R2$ and $R3$, as well as under the moves $R4$ and $R5$ for crossing-type oriented vertices. It remains to show that $[\,\, \cdot \,\,]$ is invariant under the moves $R4$ and $R5$ for alternating oriented vertices. We look first at the move $R4$:  
 \begin{eqnarray*}
 \left[\,
       \raisebox{-26pt}{{ \includegraphics[height=.8in]{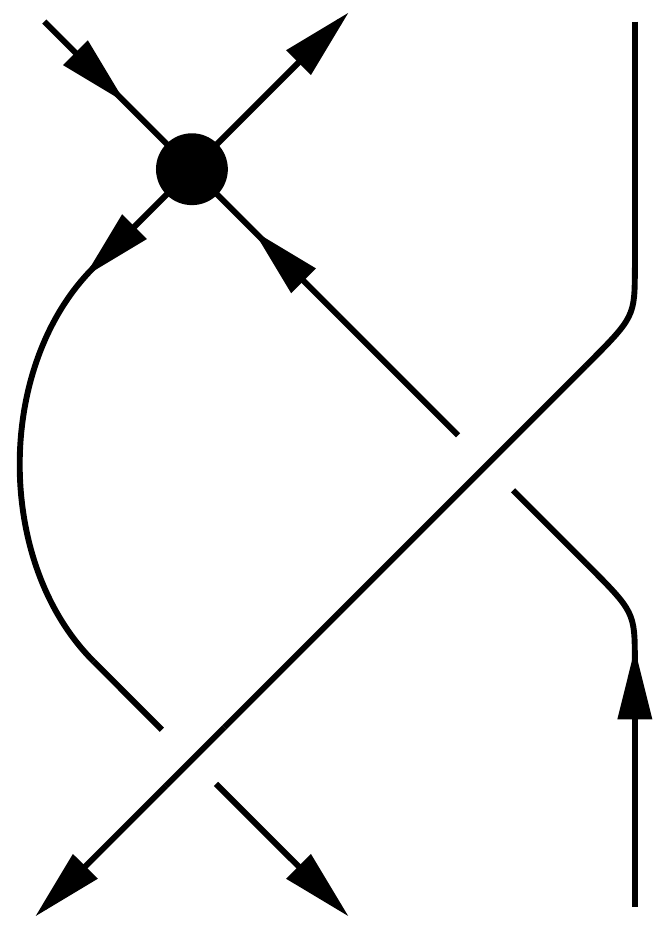}}}
 \right]
 &=& \gamma
 \left[
      \raisebox{-26pt}{ \includegraphics[height=.8in]{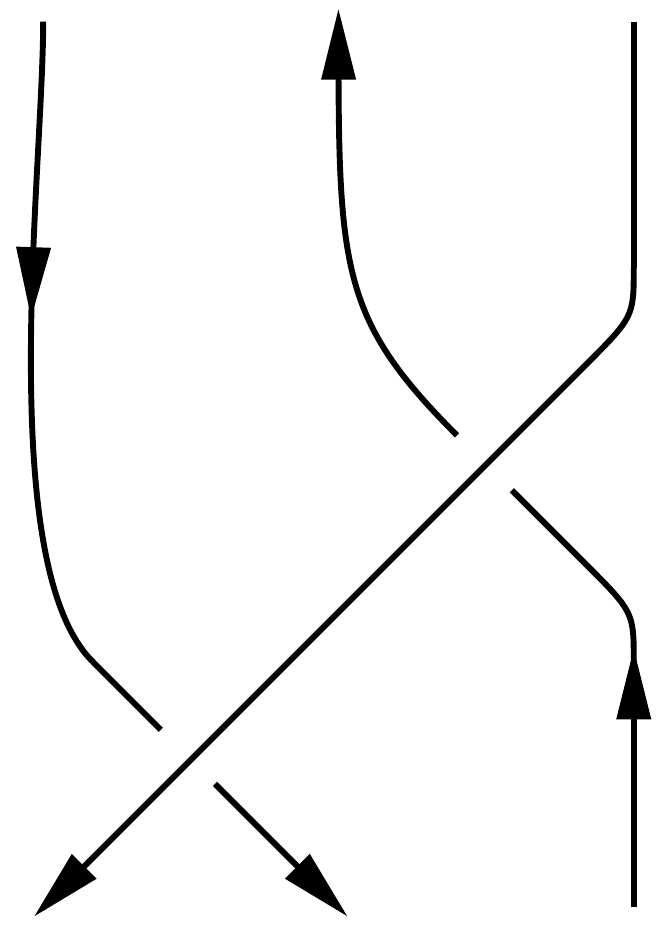}}
 \right]
 +  \gamma
  \left[
      \raisebox{-26pt}{ \includegraphics[height=.8in]{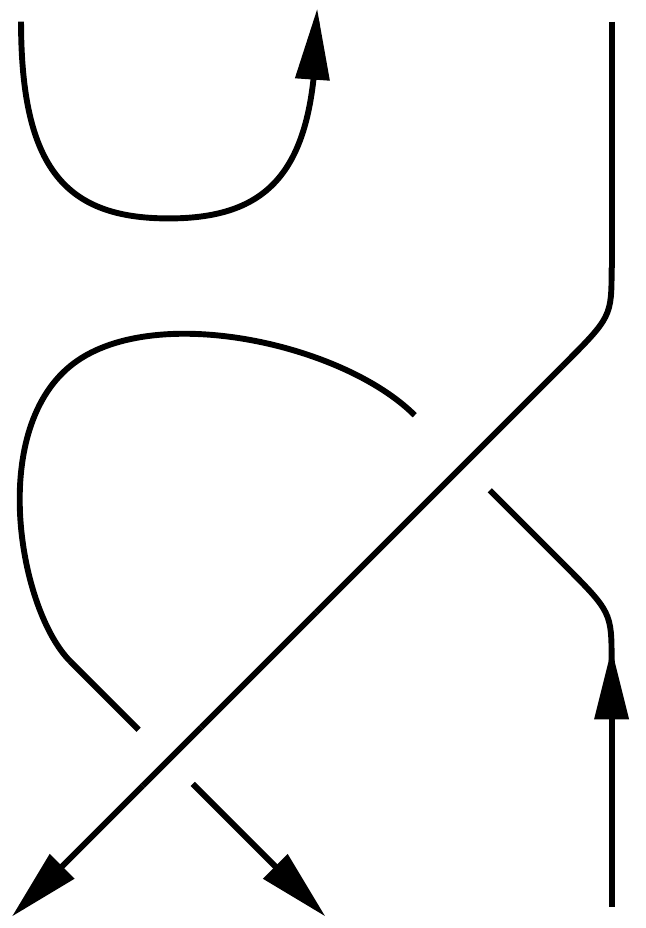}}
 \right] 
 \\
 &\stackrel{R2}{=}& \gamma
 \left[
      \raisebox{30pt}{ \includegraphics[height=.8in,angle=180]{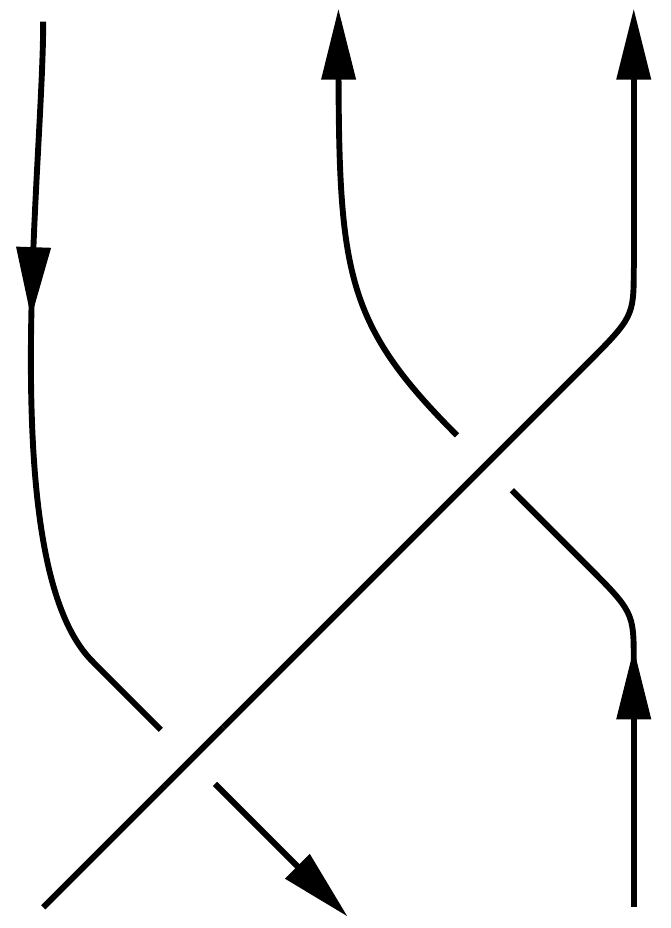}}
 \right]
 +  \gamma
  \left[
      {\raisebox{30pt}{ \includegraphics[height=.8in,angle=180]{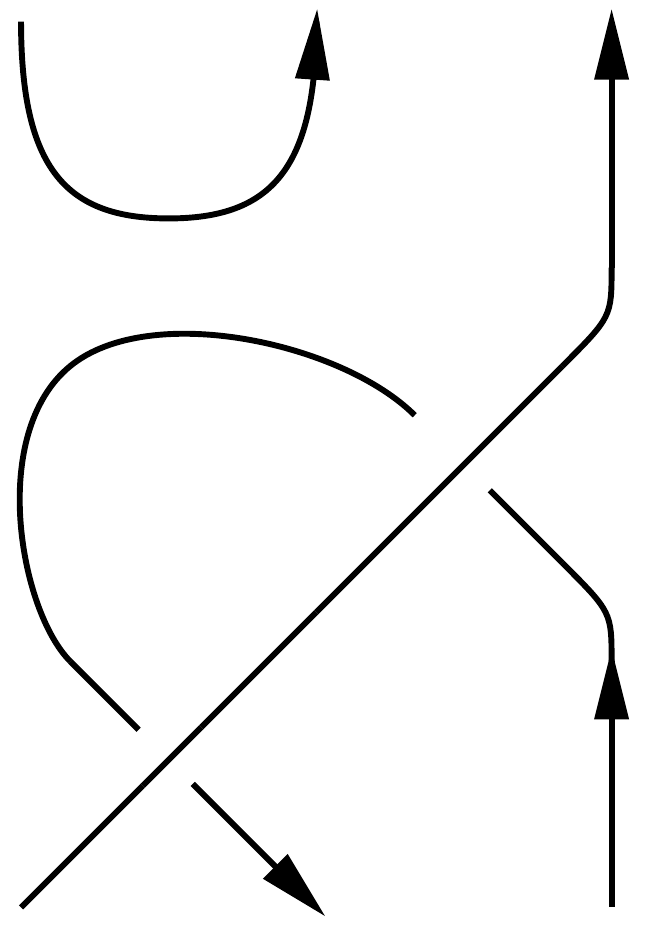}}}
 \right]
 = \left[
      \raisebox{30pt}{ \includegraphics[height=.8in,angle=180]{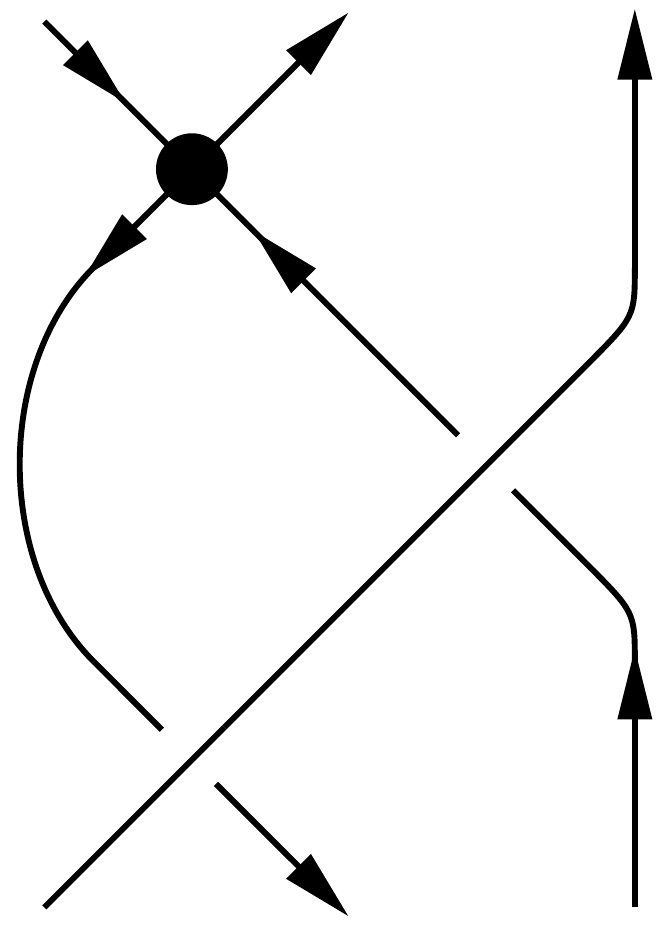}}\,
 \right].
\end{eqnarray*}

We show now the invariance of $[\,\, \cdot \,\,]$ under the move R5 for alternating oriented vertices:

\begin{eqnarray*}
 \left[\,
       \raisebox{-11pt}{{ \includegraphics[height=.35in]{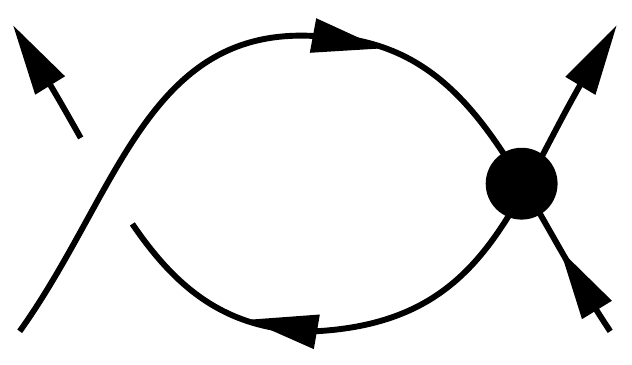}}}
 \right]
 &=& \gamma
 \left[
      \raisebox{-11pt}{ \includegraphics[height=.35in]{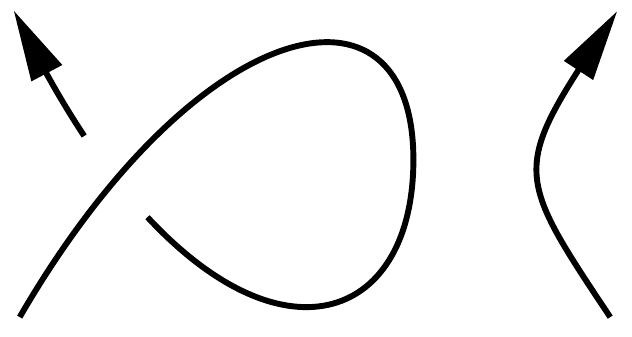}}
 \right]
 +  \gamma
  \left[
      \raisebox{16pt}{ \includegraphics[height=.4in,angle=180]{pcross}}
 \right] 
 \\
 &=& \gamma q^{n}
 \left[
      \raisebox{16pt}{ \includegraphics[height=.4in,angle=180]{2arcs}}
 \right]
 +  \gamma
  \left[
      {\raisebox{16pt}{ \includegraphics[height=.4in,angle=180]{pcross}}}
 \right]
\\
 &=& \gamma q^{n} q^{-n}
 \left[
      \reflectbox{\raisebox{-11pt}{ \includegraphics[height=.35in]{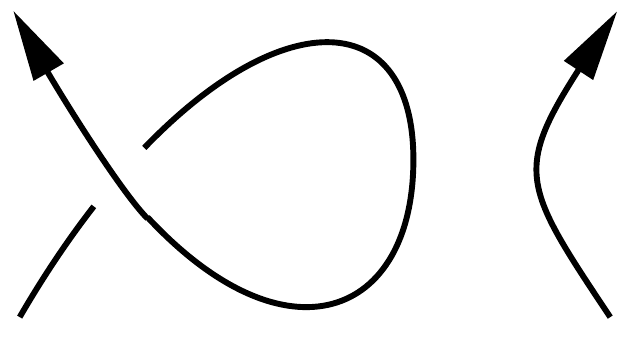}}}
 \right]
 +  \gamma
  \left[
      {\raisebox{16pt}{ \includegraphics[height=.4in,angle=180]{pcross}}}
 \right]\\
&=& \left[
     \reflectbox{ \raisebox{-11pt}{ \includegraphics[height=.35in]{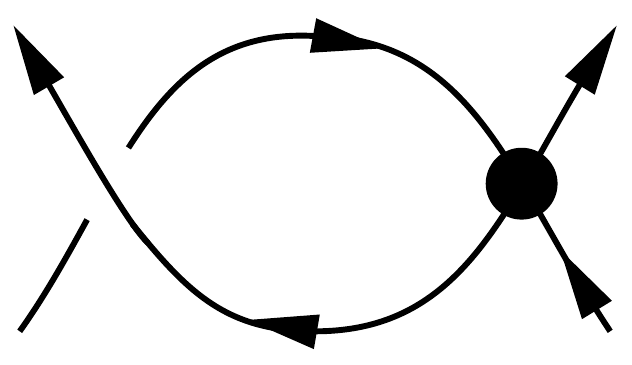}}}\,
 \right].
\end{eqnarray*}
This completes the proof.
  \end{proof}
 
The previous theorem says that $[\,\, \cdot \,\,]$ is invariant under the moves $R2, R3, R4$ and $R5$, but not yet under the move $R6$. Can we do better than this? Can we obtain an invariant for balanced oriented \textit{topological} knotted graphs?

By Proposition~\ref{prop:poly-r6}, we have
\[ \left[ \raisebox{-13pt}{ \reflectbox{\includegraphics[height=0.4in]{r61}}}\, \right]
               = q \,\left[ \raisebox{-13pt}{ \includegraphics[height=.4in]{flatcross}}\, \right]
               \hspace{1cm}
\left[  \raisebox{-13pt}{ \includegraphics[height=0.4in]{r61}}\, \right]
               = q^{-1}\left[ \raisebox{-13pt}{ \includegraphics[height=.4in]{flatcross}} \, \right].\]
Therefore, we would like the following to hold, as well:

\begin{eqnarray}\label{eq:alternating}
\left[ \reflectbox{\raisebox{-12 pt}{\includegraphics[height=0.45in, angle=90]{altr5} }}\,\right]=  q\left[\,\raisebox{-13 pt}{\includegraphics[height=.4in]{flatalt} }\,\right] \hspace{1cm} \left[\, \reflectbox{ \raisebox{-12 pt}{\includegraphics[height=0.45in, angle=90]{altr52} }}\,\right]=  q^{-1}\left[\,\raisebox{-13 pt}{\includegraphics[height=.4in]{flatalt} }\,\right].
\end{eqnarray}

Once the identities in Equation~\eqref{eq:alternating} are satisfied, we can do the following: Given $G$ a balanced oriented knotted graph diagram, let $\epsilon(G)$ be the \textit{writhe} of $G$  given by the summation of the signs of all crossings in $G$, where
\[ 
\epsilon \left(  \raisebox{-12pt}{ \includegraphics[height=.4in]{pcross}}  \right)=1, \hskip 20 pt
\epsilon \left(  \raisebox{-12pt}{ \includegraphics[height=.4in]{ncross}}  \right)=-1,\hskip 20 pt
\]
and let \[P(G)= q^{-\epsilon(G)} [G].\]
We see that if the identities in Equation~\eqref{eq:alternating} hold, then $P(G)$ is a regular isotopy invariant for balanced oriented topological knotted graphs.

Using the skein relations in Proposition~\ref{prop:useful}, we have
\begin{eqnarray*}
 q^{-1} \left[\,  \reflectbox{\raisebox{-13pt}{  \includegraphics[width=.4in,angle=90]{altr5}}} \right] 
 &=& q^{-1}\left[\, \reflectbox{\raisebox{15 pt}{\includegraphics[height=.4in, angle = 270]{reid2b}}}\,\right] + q^{-1} \cdot q \left[\, \raisebox{-11 pt}{\includegraphics[height=.35in]{poskink} } \raisebox{15 pt}{\includegraphics[height=.35in, angle = 180]{arc}}\,\right]
 \\&=& q^{-1} \left[\, \raisebox{-13 pt}{\includegraphics[height=.4in, angle = 90]{2arcs-op}}\,\right] + q^n \left[\, \raisebox{-13 pt}{\includegraphics[height=.4in]{2arcs-op2}}\,\right].
\end{eqnarray*}
Similarly, 
\begin{eqnarray*}
q \left[\, \reflectbox{ \raisebox{-12 pt}{\includegraphics[height=0.45in, angle=90]{altr52} }}\,\right] 
&=& q \left[\, \raisebox{-13 pt}{\includegraphics[height=.4in, angle = 90]{2arcs-op}}\,\right] + q^{-n} \left[\, \raisebox{-13 pt}{\includegraphics[height=.4in]{2arcs-op2}}\,\right].
\end{eqnarray*}

Imposing the equalities in~\eqref{eq:alternating}, we see that we need
\[q^n= q^{-n} = \gamma \,\, \text{and} \,\, q = q^{-1} = \gamma, \]
or equivalently, $q = \pm1$. We obtain that $P(G)_{|q = \pm1}$ is an ambient isotopy \textit{numerical invariant} for balanced oriented topological knotted graphs. 

However, if $q = \pm1$, the skein relation defining the regular isotopy version of the $sl(n)$ link polynomial and its extension to knotted graphs implies that 
\[\left[ \raisebox{-12pt}{ \includegraphics[height=.4in]{pcross}}\,\right] =  \left[\raisebox{-12pt}{ \includegraphics[height=.4in]{ncross}}\,\right]\]
and, therefore, this numerical invariant does not distinguish between different embeddings of a graph, which is rather disappointing.\\

\textit{Concluding remarks.} In this paper, we employed a solution of the Yang-Baxter equation to construct, for each integer $n \geq 2$, a polynomial invariant $\brak{\, \cdot \, }$ of regular isotopy for singular links. Then we studied some properties of the resulting polynomials. These polynomials can also be defined via the representations $\rho_n$ introduced in Section~\ref{sec:repres}. For each fixed integer $n \geq 2$, we extended further the polynomial $\brak{\, \cdot \,}$ to allow not only crossing-type oriented vertices but also alternating oriented vertices. We showed that the resulting Laurent polynomial $[\, \cdot \,]$ is an invariant of rigid-vertex regular isotopy for balanced oriented knotted graphs. In addition, in Section~\ref{sec:MOYrelations} we showed an interesting connection between our polynomial $\brak{\, \cdot \,}$ for singular links and the MOY state model for the $sl(n)$ polynomial for classical knots and links.
 
\textbf{Acknowledgements.} This research was partially completed during the 2013 Fresno State Mathematics REU Program, supported by  NSF grant \#DMS-1156273. The authors would also like to thank the referee for her/his careful reading of the paper and valuable comments and suggestions.


\end{document}